\documentclass[11pt]{amsart}
\usepackage{amsmath}
\usepackage{amssymb}
\usepackage{tabularx}
\usepackage{enumerate}
\usepackage[dvipdfm]{graphicx}
\usepackage{texdraw}

\usepackage{mathrsfs}

\usepackage{amsfonts,amssymb,amsmath}
\usepackage{epsfig}
\usepackage{xcolor}

\makeatletter
\@addtoreset{equation}{section}

\makeatother
\newtheorem{thm}{Theorem}[section]
\newtheorem{lem}[thm]{Lemma}

\newtheorem{prop}[thm]{Proposition}
\newtheorem{remark}[thm]{Remark}

\newcommand{\R}{\mathbb{R}}

\newcommand{\calS}{\mathcal{S}}
\newcommand{\X}{\mathscr{X}}
\newcommand{\V}{\mathcal{V}}

\begin{document}
\title[Fife-McLeod's Theorem for Degenerate Equations]
{Fife-McLeod's Theorem for Spatially Periodic Degenerate Diffusion Equations$^\S$}
\thanks{$\S$ This research was partly supported by the NSFC (No. 12471199, 12101413, 12071299). }
\author[B. Lou  and J. Lu]{Bendong Lou$^\dag$ and Junfan Lu$^{\dag}$}
\thanks{$\dag$ Mathematics and Science College, Shanghai Normal University, Shanghai 200234, China.}
\thanks{{\bf Emails:} {\sf lou@shnu.edu.cn} (B. Lou), {\sf jlu@shnu.edu.cn} (J. Lu)}
\date{\today}

\begin{abstract}
For one dimensional homogeneous bistable diffusion equations, Fife-McLeod (\cite{FM}) gave a well-known theorem which says that spreading solutions starting from compactly supported initial data can be exponentially approximated by traveling wave solutions. We will extend this theorem to {\it degenerate diffusion equations in periodic environments}.
First, we construct a {\it periodic traveling sharp wave} to the equation, which has a positive profile on the left half-line and a right free boundary governed by the Darcy's law. To achieve this we use a renormalization approach
in which crucial uniform gradient estimates near the free boundary are derived via delicate asymptotic analysis. Next we show that the central part of any spreading solution decays exponentially to a periodic steady state. Based on these results, we can construct super- and sub-solutions to prove the Fife-McLeod's theorem for our equation: any spreading solution with compactly supported initial data can be exponentially
approximated by the periodic traveling sharp wave.
\end{abstract}

\maketitle


\section{Introduction}
Consider the following degenerate reaction diffusion equation:
\begin{equation}\label{E0}
u_t = [A(u)]_{xx} + F(x,u), \qquad x\in \R,\ t>0,
\tag{E0}
\end{equation}
where $F$ is $L$-periodic in $x$, and $A$ belongs to the following class
\begin{equation*}\label{ass-A}
\mathcal{A}:= \left\{ A \left|
 \begin{array}{l}
 A\in C^1([0,\infty)) \cap C^\infty ((0,\infty)), \  A(u), A'(u), A''(u), A'''(u)>0 \mbox{ for }u>0,\\
A(0)=A'(0)=0,\  \int_0^1 \frac{A'(r)}{r} dr<\infty,\ \frac{rA''(r)}{A'(r)} \to A_* >0\ \mbox{ as } r\to 0+0
 \end{array}
 \right.
 \right\}.
\tag{A}
\end{equation*}
Examples of such $A$ include $A_1(u):=u^m\ (m>1)$, $A_2(u) :=u^m+u^n\ (m,n>1)$, $A_3(u):= u^m \log(u+1)\ (m>2)$, etc. In any of these cases $A(u)$ is a nonlinear diffusion term, and it is degenerate at $u=0$. The equation \eqref{E0} can be used to model population dynamics with diffusion flux depending on the population density, the combustion, propagation of interfaces and nerve impulse propagation phenomena in porous media, as well as the propagation of intergalactic civilizations in the field of astronomy (cf. \cite{A1, GN, GurMac, NS, SGM, Vaz-book, WuYin-book} etc.).

When $A(u)=u$, the equation \eqref{E0} reduces to a classical reaction diffusion equation (RDE, for short), in which the diffusion term is a linear and non-degenerate one. Such equations have been widely studied in the last decades. Among others, in the celebrated paper \cite{FM}, Fife and McLeod considered RDE with a homogeneous bistable reaction term, and proved the following well-known result:

\medskip
\noindent
{\bf Fife-McLeod's Theorem}(\cite{FM}). {\it  Assume $A(u)\equiv u$ and $F(x,u)=u(u-a)(1-u)$ for some $a\in (0,\frac12)$. Let $u$ be the solution of \eqref{E0} with compactly supported initial data $u_0(x)$. If $u(\cdot ,t)\to 1\ (t\to \infty)$ in $L^\infty_{loc}(\R)$ topology, then, there exist constants $x_+, x_-, K$ and $\omega$ (the last two positive) such that
\begin{equation}\label{FM-results}
\left\{
\begin{array}{ll}
|u(x,t)- U(x-ct-x_+)|\leq Ke^{-\omega t}, & x>0,\\
|u(x,t)- U(-x-ct-x_-)|\leq Ke^{-\omega t}, & x<0,
\end{array}
\right.
\end{equation}
where $U(x-ct)$ is the rightward traveling wave connecting $1$ and $0$.}
\medskip

\noindent
This result is precise in the sense that the shifts $x_\pm $ are bounded and the decay of $|u-U|$ is exponential. Similar results are not necessarily to be true for RDEs with other reactions. For example, in the Fisher-KPP equations, the shifts are unbounded ones involving a logarithmic term $\ln t$.

In the last decades, Fife-McLeod's theorem has exerted a very profound influence on the subsequent studies related to asymptotic behavior of solutions to RDEs.
For example, Pol\'{a}\v{c}ik \cite{Polacik-MAMS} considered \eqref{E0} with $A(u)\equiv u$ and with multistable homogeneous reactions $F=F(u)$. He used the {\it propagating terrace}  (that is, a set of traveling waves) to characterize the spreading of solutions, and proved convergence results like
\begin{equation}\label{Polacik-conv}
u(x,t)- U(x-ct+o(t)) \to 0 \quad \mbox{\ \ as\ \ }t\to \infty,
\end{equation}
in the $L^\infty_{loc}(\R)$ topology.
Ducrot, Giletti and Matano \cite{DGM} considered \eqref{E0} with $A(u)\equiv u$ and multistable spatially periodic reactions $F=F(x,u)$. They used a renormalization approach to construct {\it periodic traveling waves} (also called {\it pulsating traveling waves}) $U(x-ct,x)$, as well as {\it periodic propagating terrace}. According to their results, when $F$ is a monostable or bistable reaction, any solution $u(x,t)$ starting from the Heaviside function can be approximated by the periodic traveling wave (cf. \cite[Theorem 1.12]{DGM}):
\begin{equation}\label{DGM-results}
\|u(x,t)- U(x-ct+o(t),x)\|_{L^\infty(\R)}\to 0 \mbox{\ \ as\ \ } t\to \infty.
\end{equation}

As we have mentioned above, degenerate diffusion equations can be used to
model the density-dependent diffusion processes. In particular, the porous medium equations (PMEs, for short, that is, the case where $A(u)\equiv u^m$ for some $m>1$) with various reactions have attracted attention in the last decades (cf. \cite{A1, ACP, Biro, DQZ, Garriz, GK, LouZhou, PV, She}).
Among them, Du, Quir\'{o}s and Zhou \cite{DQZ} considered $N$-dimensional radially PME with homogeneous logistic reaction $u(1-u)$ and gave the following estimate for spreading solutions (hereinafter, by a spreading solution we mean a solution of the Cauchy problem converging to some positive stationary solution) starting from compactly supported initial data (cf. \cite[Theorem 1.1]{DQZ}):
\begin{equation}\label{DQZ-results}
\|u(|x|,t)- U(|x|-ct +(N-1)c_{\sharp} \log t -r_0)\|_{L^\infty(\R^N)}\to 0 \mbox{\ \ as\ \ } t\to \infty,
\end{equation}
where $U(x-ct)$ is the traveling sharp wave of one dimensional equation $u_t=(u^m)_{xx}+u(1-u)$ and $c_{\sharp}$ is some constant. When $N=1$, the logarithmic correction term vanishes and the estimates was already suggested by \cite{Biro}. G\'{a}rriz \cite{Garriz} extended the limit \eqref{DQZ-results} to one dimensional PME with homogenous bistable reactions.

Note that, for the multistable homogeneous RDEs, spatially periodic RDEs, as well as homogeneous PMEs, the asymptotic limits \eqref{Polacik-conv}, \eqref{DGM-results} and \eqref{DQZ-results} for the spreading solutions starting from compact supported initial data are not as precise as that in Fife-McLeod's theorem: the shifts are not bounded and/or the convergence is not exponential.
Our main purpose in this paper is to prove Fife-McLeod's theorem for degenerate diffusion equation with spatially periodic reaction \eqref{E0}, which includes the following three main parts.

(1) To construct the {\it periodic traveling sharp wave} to \eqref{E0} (Theorem \ref{thm:main1}), which means a periodic traveling wave solution having a positive profile on the left half-line and a right free boundary where the Darcy's law is satisfied. To achieve this we use a renormalization approach in which crucial uniform gradient estimates near the free boundary are derived via delicate asymptotic analysis. As far as we know, there is no periodic sharp wave for degenerate equations in heterogeneous environment has been constructed so far.

(2) To show the exponential convergence for the central part of a spreading solution to a positive periodic steady state (Theorem \ref{thm:Cauchy} (ii)). Since we are concerned with spreading solutions
starting from compactly supported initial data rather than from Heaviside functions or traveling-front-like initial date as in \cite{DGM} (for which the solution tends to some positive steady state on the left infinity), we have to split the spreading solution into right and left parts and to give the exponential decay rate for the solution to the positive stationary solution in the middle part (especially at $x=0$). This will be done in subsection 5.2 by a precise calculation for heat equations.

(3) To construct super- and sub-solutions by using the sharp wave and to prove the Fife-McLeod's theorem for \eqref{E0}: any spreading solution with compactly supported initial data can be exponentially approximated by the periodic traveling sharp wave.

\bigskip

For linear diffusion equations (that is, $A(u)\equiv u$) with multistable heterogeneous reactions, Ducrot, Giletti and Matano \cite{DGM} used the following assumption

\medskip
\noindent
{\it Assumption 1}. There exists a positive periodic stationary solution $p$, and no other periodic stationary solutions between $0$ and $p(x)$ which are both isolated from below and stable from below. There exists a solution with compactly supported initial data $0 \leq u_0(x)<p(x)$ that converges locally uniformly to $p$.

\medskip
\noindent
This assumption guarantees the existence of propagating terrace (or, traveling wave) of the equation. It is quite general in form, but is not easy to be verified. To facilitate the description, in this paper we apply a special version of $F(x,u)$. More precisely, hereafter we consider
\begin{equation}\label{E}
u_t = [A(u)]_{xx} + f(x,u)[\kappa(x) -u], \qquad x\in \R,\ t>0,
\tag{E}
\end{equation}
where
\begin{equation*}\label{F}
\left\{
\begin{array}{l}
f(x,u)\in C^2(\R^2),\ \kappa(x)\in C^2(\R),\ \mbox{both are }L\mbox{-periodic in }x,\\
f(x,0)\equiv 0, \ \exists \ \theta\geq 0 \mbox{ such that } f(x,u)>0 \mbox{ for }x\in \R,\ u>\theta, \\
\theta < \kappa_0 := \min\kappa(x)<\kappa^0 :=\max \kappa(x)<\infty,\\
\mbox{there exists } g_0(u) \in C^2([0,\infty)) \mbox{ with } g_0(0)=g_0(\kappa_0)=0\ \mbox{ such that}\\
\int_u^{\kappa_0} A'(r) g_0 (r) dr >0 \mbox{ for } u\in [0,\kappa_0), \mbox{ and }
f(x,u)[\kappa(x)-u]\geq g_0(u) \mbox{ for } x\in \R,\ u\geq 0.
\end{array}
\right.
\tag{F}
\end{equation*}
Our assumption is similar to {\it Assumption 1} when the reaction term is as in \eqref{E}, and it is clearer and easier to be verified.
Note that, when the reaction term in \eqref{E} is a homogeneous one like
$F=F(u)=u(u-\theta)(\kappa_0 -u)$, the assumption \eqref{F} is roughly equivalent to  $\int_0^{\kappa_0} A'(r)F(r)dr>0$, which is nothing but the standard condition ensuring the spreading phenomena for homogeneous degenerate equations. In this sense we can say that the assumption \eqref{F} is almost necessary. Other examples of $f$ satisfying \eqref{F} include the following cases:
\begin{enumerate}
\item[\it Case (a)] (Monostable nonlinearity): $f(x,u)>0$ for all $x\in \R$ and $u>0$;
\item[\it Case (b)] (Bistable nonlinearity): there exists $\theta>0$ such that $f(x,u)< 0$ for $x\in \R$ and $u\in (0,\theta)$, $f(x,u)>0$ for $x\in \R$ and $u>\theta$, $f_u(x,0)<0<f_u(x,\theta)$ for $x\in \R$, and \eqref{E} has a positive periodic stationary solution $\hat{p}(x)$ around $\theta$;
\item[\it Case (c)] (Combustion nonlinearity):  there exists $\theta>0$ such that $f(x,u)\equiv 0$ for $x\in \R$ and $u\in (0,\theta)$, and $f(x,u)>0$ for $u>\theta$;
\item[\it Case (d)] (Multistable nonlinearity):  there exists $\theta>0$ such that $f(x,\cdot)$ changes sign finitely many  times in $(0,\theta)$, while $f(x,u)>0$ for $u>\theta$.
\end{enumerate}
\noindent
Note that, in Case (b), the condition $f_u(x,\theta)>0$ and the existence of $\hat{p}$ are obvious if $f(x,u)[\kappa(x)-u]= F(u)$ near $u=\theta$ with $F(\theta)=0$ and $F'(\theta)>0$. They are technical conditions to be used in subsection \ref{subsec:bistable-PTW} to derive the periodic positive stationary solution $p(x)$ in Theorem \ref{thm:main1} which lies in the range  $[\kappa_0,\kappa^0]$ in the bistable case. We guess that they are not necessary.

Our first result is about the existence of periodic traveling sharp wave.

\begin{thm}[{\bf Periodic traveling sharp wave}]\label{thm:main1}
Assume \eqref{ass-A} and \eqref{F}. Then the equation \eqref{E} has a periodic traveling sharp wave $U(x,t)\in C(\R^2)$ satisfying
\begin{equation}\label{U=PTW}
U(x,t+T)\equiv U(x-L,t),\quad x,\ t\in \R,
\end{equation}
for some $T>0$. Moreover,
\begin{enumerate}[{\rm (i)}]
\item there exists $R(t)\in C^1(\R)$ such that
$$
U (x,t)>0 \mbox{ in } Q:=\{(x,t)\in \R^2 \mid x<R(t),\ t\in \R\}, \qquad U(x,t)\equiv 0\mbox{ in } Q^c := \R^2\backslash Q;\
$$
and $U\in C^{2+\alpha,1+\alpha/2}(Q)$ for any $\alpha\in (0,1)$;

\item for some $L$-periodic, positive stationary solution $p(x)$ of \eqref{E}, there holds,
$$
U(x,t)-p(x)\to 0\mbox{ as }x\to -\infty, \qquad \mbox{for all } t\in \R;
$$

\item there exists $\delta^*>0$ such that the Darcy's law holds in the following sense
$$
R'(t)= - \lim\limits_{x\to R(t)-0} \left[\frac{A'(U(x,t))}{U(x,t)} U_x (x,t) \right]\geq \delta^*,\qquad t\in \R.
$$
\end{enumerate}
\end{thm}

According to this theorem, $U(x,t)$ is a sharp wave with a right free boundary where the Darcy's law is satisfies as in (iii) for the generalized pressure function (see details in section 2). As far as we know, there are no much results on {\it periodic traveling sharp waves} for {\it degenerate diffusion equations in heterogeneous environments}. Even for monostable, bistable or combustion equations.

This theorem will be proved by the renormalization method. More precisely, we will take a positive and periodic stationary solution $p_0(x)\in [\kappa_0,\kappa^0]$ and consider the Cauchy problem of \eqref{E} with initial data
\begin{equation}\label{ini-PME}
u_0(x) = p_0(x)H(-x),\qquad x\in \R,
\end{equation}
where $H(x)$ is the Heaviside function. Denote the solution of this problem by $u(x,t)$, its unique free boundary by $r(t)$. For any positive integer $n$, define $t_n$ as the time when $u(nL,t_n)=0$. Then the periodic sharp wave $U$ and its free boundary $R(t)$ in the previous theorem can be obtained as in the following result.

\begin{prop}\label{prop:PTW-limit}
Assume \eqref{ass-A} and \eqref{F}. Let $u(x,t),\ r(t)$ and $t_n$ be as above. Then
\begin{equation}\label{u-to-U}
u(x+nL,t+t_n)\to U(x,t) \mbox{\ \ as\ \ }n\to \infty,
\end{equation}
in the topology $L^{\infty}_{loc}(\R^2)$ and $C^{2,1}_{loc}(Q)$, and
\begin{equation}\label{b-to-B}
r(t+t_n) -n L\to R(t) \mbox{\ \ as\ \ }n\to \infty,
\end{equation}
in the topology $L^\infty_{loc}(\R)$.
\end{prop}

\noindent
We will show in section 3 that the sequence $\{u(x+nL,t+t_n)\}$ is decreasing in $n$, so, it is a big problem to prevent the limit function $U$ from the trivial solution $0$. For this purpose, we will give uniform positive lower bounds for the instantaneous speed $r'(t)$ and for $u(x,t)$ when $x$ lies on the left neighborhood of the free boundary. These estimates constitute the main issue and difficulties in section 3.

Note that the reaction term in \eqref{E} can be a multistable one, and the equation may have a propagating terrace. This is indeed possible for linear diffusion equations with multistable reactions (cf. \cite{DGM}). In such cases, the periodic sharp wave we obtained in the previous theorem is indeed the lowest traveling wave in the terrace. So, the limit $p(x)$ of $U(x,t)$ as $x\to -\infty$ may be smaller than $\kappa_0$ and smaller than $p_0(x)$, though the renormalized sequence in the previous proposition is given by the solution starting at $u_0$ in \eqref{ini-PME}. For the equations with simpler reactions, however, we can expect that $p(x)$ is a solution in
$$
\calS:=\{\tilde{p}\mid \tilde{p}\mbox{ is an }L\mbox{-periodic stationarty solution of \eqref{E} with range in }[\kappa_0,\kappa^0]\}.
$$

\begin{prop}\label{prop:p>kappa}
Assume \eqref{ass-A} and \eqref{F}. If $f$ is of monostable, bistable, or combustion type as in Case (a), (b) or (c), then the periodic stationary solution $p(x)$ in Theorem \ref{thm:main1} (ii) lies in $[\kappa_0,\kappa^0]$.

In addition, the convergence in Theorem \ref{thm:main1} (ii) can be strengthened as follows: there exist $M_*,\delta_*>0$ such that
\begin{equation}\label{U-to-p-exp}
0< p(x)- U(x,t) \leq M_* e^{\delta_* (x-c^* t)},\qquad x,t\in \R,
\end{equation}
where $c^*:= \frac{L}{T}$ is the average propagating speed of $U$.
\end{prop}

\medskip
One of the main purposes of seeking for the periodic traveling sharp wave is to use it to characterize the spreading solutions of \eqref{E} with compactly supported initial data. For simplicity, we choose initial data from the following set
\begin{equation}\label{cond of initial}
\X:= \Big\{ w\in C(\R) \; \left| \;
 \exists\ b>0 \mbox{ such that }  w(x) >0 \mbox{ for } |x|<b, \mbox{ and } w(x) =0 \mbox{ for } |x|\geq b\right. \Big\}
\end{equation}
The weak solution of such a problem is defined as the following.
For any $T>0$, denote $Q_T:= \R\times (0, T)$. We call a function $u(x, t)\in C(Q_T)\cap L^{\infty}(Q_T)$ a {\it very weak solution} of \eqref{E} with initial data $u(x,0)\in \X$ if, for any $\phi \in C^{\infty}_c(Q_T)$, there holds
$$
\int_{\R} u(x, T)\phi(x, T) dx= \int_{\R} u_0(x)\phi(x, 0) dx= \iint_{Q_T} f(x, u)\phi dx dt+ \iint_{Q_T} [u\phi_t+ A(u) \phi_{xx}] dx dt.
$$
The global well-posedness of \eqref{E} with $u(x,0)\in \X$ can be found in \cite{ACP, LouZhou, PV, Sacks, Vaz-book, WuYin-book} etc. Among others, $u(x, t)$ is continuous, and it has exactly two free boundaries $l(t)<r(t)$ which move leftward and rightward, respectively.
In addition, $u(x,t)$ is positive and classical in $(l(t), r(t))$ and is zero outside this interval.

In what follows, we say that a solution $u$ of \eqref{E} with $u(x,0)\in \X$ is a {\it spreading solution} or {\it spreading happens for $u$}, if $u(x,t)\to p(x)$ for some $p\in \calS$ in the $C^{2}_{loc}(\R)$ topology. This is a particular interesting case in the asymptotic behavior, since it describes a successful invade or spreading process for a new species in ecology problems.  We will see in section 2 that, under the condition \eqref{F}, spreading does happens for the solutions starting from large initial data. Our second main purpose in this paper is to use the traveling sharp wave to characterize the spreading solutions. It is quite complicated to do this for equation \eqref{E} with generic multistable reactions since propagating terrace may occur. So we focus on monostable, bistable and combustion equations. Furthermore, we impose some additional conditions:
\begin{equation}\label{F1}
\frac{\partial }{\partial u} \big( f(x,u)[\kappa(x)-u] \big) <0,\qquad x\in \R,\ u\in [\kappa_0, \kappa^0], \tag{F1}
\end{equation}
which is used to guarantee the uniqueness of the elements in $\calS$;
\begin{equation}\label{H}
uA''(u)\cdot \int_0^u \frac{A'(r)}{r} dr \geq [A'(u)]^2 \mbox{ for } u>0,
\qquad \mbox{and} \qquad \kappa^0 -\kappa_0\ll 1.
\tag{H}
\end{equation}
Condition \eqref{F1} is satisfied for homogeneous reactions, and \eqref{H} is satisfied automatically for $A_1(u)=u^m$ and $A_2(u)=u^m +u^n$.
They are used in section 5 (Lemma \ref{lem:increasing}) to guarantee the increasing property for $\frac{B(v)}{v}$ and $-\frac{h(x,v)}{vB(v)}$ (the latter is a Fisher-KPP type of condition) for the pressure function $v$ in a domain corresponding to $u\in [\kappa_0,\kappa^0]$, which are necessary in constructing super- and sub-solutions for spatially heterogeneous equations.

Based on the previous results we can prove the Fife-McLeod' theorem for \eqref{E}:

\begin{thm}[{\bf Fife-McLeod's theorem for periodic degenerate diffusion equations}]\label{thm:Cauchy}
Assume \eqref{ass-A}, \eqref{F} and \eqref{F1}, $f$ is a monostable, bistable or combustion type of reaction as in Case (a), (b) or (c). Let $U(x,t)$ be the rightward moving periodic sharp wave with free boundary $R(t)$. Let $u(x,t)$ be a spreading solution of \eqref{E} with initial data in $\X$, and with left and right free boundaries: $l(t)<r(t)$. Then, the following conclusions hold.

\begin{enumerate}
\item[\rm (i)] $\calS$ is a singleton $\{p\}$.

\item[\rm (ii)] For any small $c>0$ there exist $M, \delta > 0$ such that
\begin{equation}\label{exp-conv}
\max\limits_{|x|\leq ct}  |u(x, t) - p(x)|  \leq M e^{-\delta t},\qquad t>0.
\end{equation}

\item[\rm (iii)] Under the additional condition \eqref{H}, there exist $t^*, t_*\in\R$ such that, as $t\to \infty$, there holds
\begin{equation}\label{u-profile-right}
    r(t)-R(t+t^*) \to 0,\qquad
 \sup_{x\geq 0}|u(x,t)-U(x, t+t^*)|\to 0,
\end{equation}
\begin{equation}\label{u-profile-left}
   l(t)+R(t+t_*)\to 0,\qquad
 \sup_{x\leq 0}|u(x,t)-U(-x, t+t_*)|\to 0.
\end{equation}
Furthermore, there exist $K, \omega_1, \omega_2 >0$ such that
\begin{equation}\label{right-exp}
|u(x,t)-U(x, t+t^*)|\leq Ke^{-\omega_1 (R(t+t^*)-x)} + Ke^{-\omega_2 t},\qquad x\geq 0,\ t\gg 1,
\end{equation}
\begin{equation}\label{left-exp}
|u(x,t)-U(-x, t+t_*)|\leq Ke^{-\omega_1 (x-l(t+t_*))} + Ke^{-\omega_2 t},\qquad x\leq 0,\ t\gg 1.
\end{equation}
\end{enumerate}
\end{thm}

\begin{remark}\label{1-rem-iii-new}
\rm
The conclusion (iii) is an analogue of the original Fife-McLeod's Theorem, but for degenerate diffusion equations in periodic environments.
As far as we know, the exponential estimates as in \eqref{right-exp} and \eqref{left-exp} have never being reported for heterogeneous equations, even for linear (i.e., non-degenerate) diffusion equations.
\end{remark}

\begin{remark}\label{MR-rem1}\rm
The estimate in \eqref{exp-conv} shows the exponential decay in the middle of the domain. It is proved in section 5 by a precise calculation for heat equations. This kind of result seems also new for heterogeneous equations. Since it is necessary in constructing super- and sub-solutions to prove (iii) (see details in section 5), its absence may explain why so far there has no similar exponential convergence results as in (iii) for heterogeneous equations.
\end{remark}

\begin{remark}\rm
Our equation involves a degenerate and nonlinear diffusion term $[A(u)]_{xx}$. Comparing with linear and non-degenerate diffusion equations as in \cite{DHZ, DGM, DuLou, DuMatano, FM} etc., this kind of degeneracy introduces additional challenges, both in constructing periodic sharp waves and in deriving estimates \eqref{exp-conv}-\eqref{left-exp}. The conclusions in (iii) are actually proved for the corresponding pressure function $v$ (see section 2), since $v$ satisfies the Darcy's law on its free boundaries and is more regular than $u$.
\end{remark}

\begin{remark}\rm
Note that the conclusions in this theorem is for a general spreading solution $u$, rather than the special solution in Proposition \ref{prop:PTW-limit} where the solution starts from a Heaviside initial data \eqref{ini-PME}. Hence the convergence of $r(t)$ in \eqref{u-profile-right} is different from that in \eqref{b-to-B}.
\end{remark}

In summary, this paper has the following highlights:
\begin{itemize}
\item Construction of periodic traveling sharp waves for degenerate diffusion equations in periodic environments (Theorem \ref{thm:main1}).
\item Clear description on the periodic sharp waves for monostable, bistable and combustion (Proposition \ref{prop:p>kappa}).
\item Utilization of sharp waves to give the Fife-McLeod's theorem for {\it heterogeneous} and {\it degenerate} diffusion equations (Theorem \ref{thm:Cauchy}).
\end{itemize}

\medskip
This rest of the paper is arranged as follows. In section 2, we present some preliminaries: define a generalized pressure function corresponding to the density function $u$, collect some basic results on the weak solutions, specify the stationary solutions and studying sufficient conditions for spreading phenomena. In section 3, we establish complex yet crucial estimates for solutions of Cauchy problems and then construct the periodic traveling sharp wave via the renormalization method. In section 4, we investigate further properties of traveling sharp waves for equations with monostable, bistable and combustion reactions. In section 5, we employ the sharp wave to characterize the spreading solutions, and prove Theorem \ref{thm:Cauchy}. Specifically, we first construct periodic traveling waves with compact support (apparently novel in prior literature), then use them to prove Theorem \ref{thm:Cauchy} (ii), which underpins the estimates in (iii) for two truncation functions of the spreading solutions.


\section{Preliminaries}
Before constructing the periodic sharp wave, we give some preliminary results in this section.

\subsection{Generalized pressure function}
In the study of PMEs, it is convenient to consider the pressure function $u^{m-1}$ instead of the original density function $u$ (cf. \cite{A1,Vaz-book}). Here we also apply a generalized pressure function $v(x,t)$ which is defined by
\begin{equation}\label{def-pressure}
v(x,t) = \Psi(u(x,t)) := \int_0^{u(x,t)} \frac{A'(r)}{r} dr\mbox{\ \ for } u(x,t)\geq 0.
\end{equation}
Denote the inverse function of $\Psi$ by $\psi$, then
$$
u=\psi(v):=\Psi^{-1}(v),\quad \Psi (0)=\psi (0)=0,\quad \psi(v)>0 \mbox{ for } v>0, \quad \Psi(u)>0 \mbox{ for } u>0,
$$
and $\psi,\Psi\in C([0,\infty))\cap C^\infty ((0,\infty))$.  Next, we define
\begin{equation}\label{def-B}
B(v) := A'(\psi(v)),\qquad v\geq 0.
\end{equation}
Then
$$
B(0)=0 \mbox{\ \ and\ \ } B(v) =\frac{A'(u)}{u} \cdot \psi(v) = \frac{\psi(v)}{\psi'(v)} \mbox{ for }v>0.
$$
Note that, by \eqref{ass-A} there holds,
\begin{equation}\label{B'=A*}
B'(0+0) =\lim\limits_{v\to 0+0} \frac{A'(\psi(v))}{v} =
\lim\limits_{u\to 0+0} \frac{A'(u)}{\Psi(u)} =
\lim\limits_{u\to 0+0} \frac{uA''(u)}{A'(u)} = A_*,
\end{equation}
where $A_*$ is the constant in \eqref{ass-A}. Hence $B\in C^1([0,\infty))\cap C^\infty((0,\infty))$ and $B'(v)>0$ for $v\geq 0$. Finally, we define
\begin{equation}\label{def-h}
h(x,v) := \left\{
 \begin{array}{ll}
\displaystyle \frac{f(x,\psi(v))(\kappa(x)- \psi(v))}{\psi'(v)}, & v>0,\\
0,& v=0.
\end{array}
\right.
\end{equation}
Using a similar derivation as in \eqref{B'=A*} one can show that
$$
h_v(x,0+0) = \kappa(x) f_u(x,0+0) A_*,\qquad x\in \R.
$$
Hence, $h\in C^1(\R\times [0,\infty))\cap C^\infty(\R\times (0,\infty))$.
Using $v$ we can rewrite the Cauchy problem of \eqref{E} with initial data $u_0$ as
\begin{equation}\label{p-v}
\left\{
\begin{array}{ll}
v_t = B(v) v_{xx} + v^2_x + h(x,v),&  x\in \R,\ t>0,\\
v(x,0)= v_0(x) := \Psi(u_0(x)), & x\in \R.
\end{array}
\right.
\end{equation}

\subsection{Basic properties}
Let $u(x,t)$ be a very weak solution of \eqref{E} with initial data in $\X$. We list some of its properties, which are well known for PMEs with or without reactions, and can be derived similarly for the general degenerate equation \eqref{E} (see, for example, \cite{Knerr,LouZhou,Vaz-book}) and references therein).

\begin{enumerate}[(a).]
\item {\it Free boundaries}. A right free boundary $r(t)$ and a left one $l(t)$ appear in the solution,  $u(x,t)$ is positive and classical in $(l(t),r(t))$, and is zero outside of this interval.

    If the initial data $u_0$ is given by \eqref{ini-PME}, then the solution also has a unique right free boundary $r(t)$, $u\in C(\R\times (0,\infty))$ and it is positive and classical in $\{(x,t)\mid x<r(t),\ t>0\}$.

\item {\it Waiting time and the Darcy's law}. The right/left free boundary has a waiting time $t_*(b)/t_*(-b)$ $\in [0,\infty]$, and $t_*(b)=0$ if $v'_0 (b-0)<0$. After the waiting time, the free boundary satisfy the Darcy's law:
\begin{equation}\label{general-sol-Darcy}
r'(t) = - v_x  (r(t)-0,0) \mbox{\ for\ } t>t_*(b),\qquad l'(t) = - v_x (l(t)+0,0)\mbox{ for } t>t_*(-b).
\end{equation}

\item {\it Positivity persistence and regularity}. For any $x_1\in \R$, once $u(x_1,t_1)>0$ for some $t_1\geq 0$, then $u(x_1,t)>0$ for all $t\geq t_1$, $u(x,t)$ is classical in a neighborhood of $(x_1,t_1)$ if $u(x_1,t_1)>0$.
\end{enumerate}

\subsection{Positive stationary solutions}
In this part, we consider the $L$-periodic stationary solutions of \eqref{E} with values in $[\kappa_0,\kappa^0]$. Denote the set of such solutions by $\mathcal{S}$.
Since the reaction term in \eqref{E} is a heterogenous one, the structure of $\mathcal{S}$ is generally not as clear as that in homogeneous equations.

\begin{lem}\label{lem:ppss}
Assume \eqref{ass-A} and \eqref{F}. Then $\mathcal{S}$ is not empty. It has a minimal element $p_0(x)$ and a maximal element $p^0(x)$.
\end{lem}

\begin{proof}
Consider the solution $u(x,t;\frac{\kappa_0+\theta}{2})$ of \eqref{E} with initial data $\frac{\kappa_0+\theta}{2}$, where $\kappa_0$ and $\theta$ are the constants in \eqref{F}. Then $u(x,t;\frac{\kappa_0+\theta}{2})$ exists globally, bounded from above by $\kappa^0$, and is increasing in $t$. Thus $u(x, t;\frac{\kappa_0+\theta}{2})\to p_0(x)$ as $t\to \infty$ pointwisely, for some function $p_0(x)\in [\kappa_0,\kappa^0]$. Since $u(x, t;;\frac{\kappa_0+\theta}{2})\geq \frac{\kappa_0+\theta}{2}$, it is actually a classical solution, and so the convergence $u\to p_0$ holds in $C_{loc}^{2}(\R)$ topology. This implies that $p_0(x)\in \mathcal{S}$, and it is the minimal element in $\mathcal{S}$.
In addition, since $f$ and $\kappa$ are $L$-periodic in $x$, the equation is invariant when we replace $x$ by $x + L$, we see that $u(x,t; \frac{\kappa_0+\theta}{2})\equiv u(x+L,t; \frac{\kappa_0+\theta}{2})$, and so $p_0(x+L)\equiv p_0(x)$.

To obtain $p^0(x)$ we only need to take limit as $t\to \infty$ for the solution $u(x, t; \kappa^0+1)$.
\end{proof}

Under some further conditions we can see that $\mathcal{S}$ is a singleton.

\begin{lem}\label{lem:unique-S}
Assume \eqref{ass-A}, \eqref{F} and \eqref{F1}.
Then $\mathcal{S}$ is a singleton.
\end{lem}

\begin{proof}
The condition \eqref{F1} is a Fisher-KPP type one. Using it one can derive a contradiction by the maximum principle if $\mathcal{S}$ has a minimal element $p_0$ and a maximal element $p^0$ with $p_0(x)\leq,\not\equiv p^0(x)$ (as showing the uniqueness of positive solutions for Fisher-KPP equations).
\end{proof}

\subsection{Spreading conditions}
The reaction term in \eqref{E} is a generic one. It includes the bistable and multistable cases. Therefore, without further conditions, vanishing phenomena (that is, $u\to 0$) and propagating terrace may occur.
We now discuss sufficient conditions on $f$ and $u_0$ which ensure the spreading phenomena.

Let $g_0(u)$ be the function in \eqref{F}. We consider the following homogeneous equation.
\begin{equation}\label{homo-PME-eq}
u_t = [A(u)]_{xx} + g_0(u).
\end{equation}
If there exist $c>0$ and a function $\varphi(z)\in C^2(J)$ such that $\varphi (z)>0$ in $J$, $\varphi(z)=0$ in $\R\backslash J$, and
\begin{equation}\label{TW-q}
[A(\varphi)]'' + c\varphi' + g_0(\varphi)=0,\qquad z\in J,
\end{equation}
then we call $u=\varphi(x-ct)$ a traveling wave of the equation \eqref{homo-PME-eq}.
Note that it is only a traveling-wave-type solution of the equation \eqref{homo-PME-eq} in $Q':=\{(x,t) \mid x-ct \in J\}$, but maybe not a solution of the Cauchy problem of \eqref{homo-PME-eq} since it does not necessarily satisfy the Darcy's law on the boundaries of $Q'$.

We consider \eqref{TW-q} on the $[\Psi(u)][\Psi(u)]'$-phase plane.
By the phase plane analysis as in \cite{A1, AW1, LiLou} we see that,
under the assumption
\begin{equation}\label{g0-ass}
\int_u^{\kappa_0}A'(r)g_0(r)dr>0,\qquad u\in [0,\kappa_0),
\end{equation}
the equation \eqref{TW-q} with $c=0$ has two trajectories which connect the singular point $(\Psi(\kappa_0),0)$ to the positive and negative $[\Psi(u)]'$-axis, respectively. By the continuous dependence of the solutions of \eqref{TW-q} on $c$ and its initial value, we can show the following result.
\begin{prop}\label{prop:compact-TW}
Assume \eqref{g0-ass}. Then for any given small $c\geq 0$ and any $d\in (0,\kappa_0)$ sufficiently close to $\kappa_0$, there exists $l_*(c,d)>0$ such that, for any $l\geq l_*$, the equation \eqref{TW-q} has a solution $\varphi(z;c)$ which satisfies
\begin{equation}\label{prop-q}
\varphi(z;c)>0 \mbox{ in } (-l,l),\quad \varphi(\pm l;c)=0,\quad \max\limits_{z\in [-l,l]}\varphi(z;c)\in (d,\kappa_0),
\end{equation}
Moreover, when $c\geq 0$ is small, there holds
\begin{equation}\label{homo-sub-TW}
c< - [\Psi(\varphi)]_z (l-0).
\end{equation}
\end{prop}

\begin{remark}\label{rem:easy-spread}
\rm
The inequality \eqref{homo-sub-TW} implies that $\varphi(x-ct;c)$ is indeed a rightward traveling-wave-type subsolution of the Cauchy problem of \eqref{homo-PME-eq}. Similarly, leftward traveling-wave-type subsolution $\tilde{\varphi}(x+ct;c)$ can be find.
\end{remark}

Now we give a sufficient condition for the spreading phenomena.

\begin{thm}\label{thm:spreading}
Assume \eqref{ass-A} and \eqref{F}. Let $u(x,t;u_0)$ be the solution of \eqref{E} with initial data $u_0\in \X$. If there exists a solution $\varphi(z;c)$ of \eqref{TW-q} satisfying \eqref{prop-q} and \eqref{homo-sub-TW}, and
\begin{equation}\label{cond-ini}
u_0(x) \geq \varphi(x;c),\qquad x\in [-l,0],
\end{equation}
then spreading happens for $u(x,t;u_0)$.
\end{thm}

\begin{proof}
Denote the solution of the equation \eqref{homo-PME-eq} with initial data $u_0$ by $\tilde{u}(x,t;u_0)$. Then, by comparison we have
\begin{equation}\label{u>tilde-u}
u(x,t;u_0)\geq \tilde{u}(x,t;u_0)\geq \varphi(x-ct;c),\qquad x\in \R,\ t\geq 0.
\end{equation}

Now we recall some results in \cite{LouZhou}, where Lou and Zhou studied the Cauchy problem of a PME with a reaction:
\begin{equation}\label{RPME-LZ}
\left\{
\begin{array}{ll}
\hat{u}_t = (\hat{u}^m)_{xx} + g_0(\hat{u}),& x\in \R,\ t>0,\\
\hat{u}(x,0)=u_0(x)\in \X, & x\in \R,
\end{array}
\right.
\end{equation}
for some $m>1$. In \cite[\S 3.2]{LouZhou}, the authors proved that the solution $\hat{u}(x,t)$, with left free boundary $\hat{l}(t)$ and right free boundary $\hat{r}(t)$, is strictly increasing in $(\hat{l}(t),-b)$ and strictly decreasing in $(l,\hat{r}(t))$.
In addition, they proved a general convergence result: $\hat{u}(x,t)$ converges as $t\to \infty$ to a nonnegative stationary solution, which is a zero of $g_0(u)$, or some ground state solution.
Following a similar argument in that paper, it is not difficult to show that the results remain true for our equation \eqref{homo-PME-eq} with a generic degenerate diffusion term $[A(u)]_{xx}$.
In particular, $\tilde{u}(x,t)$ converges as $t\to \infty$ to some stationary solution $\tilde{w}(x)$ of \eqref{homo-PME-eq}, and $\tilde{w}(x)$ is non-increasing on the right side of the initial support $[-b,b]$.

Combining these properties with the second inequality in \eqref{u>tilde-u} we conclude that $\tilde{w}(x)\geq d=\max\varphi(z;c,d)$ for $x\gg b$. Thus, $\tilde{w}(x)$ must be a positive zero of $g_0(u)$ not less than $d$.
By \eqref{u>tilde-u} again  we have
$$
\liminf\limits_{t\to \infty} u(x,t)\geq d,\qquad x\in \R.
$$
Note that the convergence of $u(x,t)$ to a stationary solution $w(x)$ of \eqref{E} can also be proved as in \cite[Theorem 1.1]{LouZhou}. Thus $w(x)\geq d$. Since $d$ can be chosen close to $\kappa_0$, by the assumptions in \eqref{F} we finally conclude that the limit $w(x)$ of $u(x,t)$ is an element in $\mathcal{S}$. This proves that spreading happens for $u$.
\end{proof}


\section{Construction of the Periodic Traveling Sharp Wave}
In this section, we assume \eqref{ass-A} and \eqref{F}, and construct the periodic traveling sharp wave by using the renormalization method. 
For simplicity of notation, in this section
\begin{center}
{\it we temporarily regard the period $L$ of $f(\cdot,u),\ \kappa(x)$ and $h(\cdot,v)$ as $1$.}
\end{center}
Of course, all the discussion holds true for general period $L$.

\subsection{Intersection number}
The so-called zero number argument (cf. \cite{Ang}) developed in 1980s is a powerful tool in the study of asymptotic behavior of solutions of one-dimensional parabolic equations.
In this paper we will use the corresponding version for degenerate equations.
Hereinafter, we will use $v,\ v^{(i)},\ v_i,\ V,\ \cdots$ to
denote the corresponding pressure function of $u,\ u^{(i)},\ u_i,\ U,\ \cdots$.
We recall that if $u$ solves the equation in \eqref{E}, then $v$ solves
\begin{equation}\label{v-PME}
v_t = B(v) v_{xx} + v^2_x + h(x,v), \qquad  x\in \R,\ t>0.
\end{equation}

We first specify some special relationship between two solutions. For $i=1,2$, assume $l^{(i)}(t) <r^{(i)}(t)$ are continuous functions for $t\in [t_1, t_2)$,
$$
Q^{(i)}:= \{(x,t)\mid l^{(i)}(t)<x<r^{(i)}(t),\ t_1<t<t_2\},
$$
$v^{(i)}(x,t)\in C(\R\times [t_1, t_2)) \cap C^{2,1}(Q^{(i)})$ is a solution of \eqref{v-PME}. Denote
$$
l(t) := \max\{l^{(1)}(t), l^{(2)}(t)\},\qquad r(t):=\min\{r^{(1)}(t), r^{(2)}(t)\},\qquad t\in [t_1,t_2).
$$
When $l(t)< r(t)$ we use the following notations.

\medskip
\noindent
\underline{\it Case 1}. 
Denote $v^{(1)}(\cdot,t_0) \vartriangleright v^{(2)}(\cdot,t_0)$, if there exists $x_0\in (l(t_0), r(t_0))$ such that
$$
v^{(1)}(x,t_0)> v^{(2)}(x,t_0) \mbox{ for } l(t_0)\leq x<x_0,\qquad v^{(1)}(x,t_0)<v^{(2)}(x,t_0) \mbox{ for } x_0 < x \leq r(t_0).
$$
(Note that this notation is a little different from the concept {\it $v^{(1)}$ is steeper than $v^{(2)}$} in \cite{DGM}. The later will be used as it was in subsection 3.3.)

\medskip
\noindent
\underline{\it Case 2}. 
Denote $v^{(1)}(\cdot,t_0) \succapprox v^{(2)}(\cdot,t_0)$, if $l^{(1)}(t_0)\leq l^{(2)}(t_0) < r^{(2)}(t_0) =r^{(1)}(t_0)$, or, $l^{(1)}(t_0)=l^{(2)}(t_0) < r^{(2)}(t_0) \leq r^{(1)}(t_0)$, and
$$
v^{(1)}(x,t_0)> v^{(2)}(x,t_0) \mbox{ for } l(t) < x <  r(t).
$$

\medskip
\noindent
\underline{\it Case 3}. 
Denote $v^{(1)}(\cdot,t_0) \succ v^{(2)}(\cdot,t_0)$, if $l^{(1)}(t_0)<l^{(2)}(t_0) < r^{(2)}(t_0) <r^{(1)}(t_0)$, and
$$
v^{(1)}(x,t_0) > v^{(2)}(x,t_0) \mbox{ for } l(t) < x < r(t).
$$

\medskip
In this section we actually consider such relationships among three solutions. More precisely, let $v_i(x,t)\in C(\R\times [t_1, t_2)) \cap C^{2,1}(Q_i)$\ $(i=1,2,3)$ be three solutions of the Cauchy problem of \eqref{v-PME}, where
$$
Q_i:= \{(x,t)\mid l_i(t)<x<r_i(t),\ t_1<t<t_2\},
$$
and $l_i(t)<r_i(t)$ are free boundaries of $v_i$.
Assume further that $l_1(t) \equiv l_2(t)\equiv -\infty$, other free boundaries are bounded functions.
The original zero number diminishing properties (cf. \cite{Ang}) say that, roughly, the zero number of a solution of a {\it one-dimensional linear uniformly} parabolic equation is decreasing. Correspondingly, the {\it intersection number} of two solutions of a nonlinear uniformly parabolic equation is also decreasing (since the difference function of these two solutions solves a linear equation). In \cite{LouZhou} the authors extended this property from uniformly parabolic equations to PMEs. Their results actually hold for degenerate equations like \eqref{v-PME}.
As a consequence we have the following result (see details in \cite{LouZhou}).

\begin{prop}\label{prop:inter}
Let $v_i$ be the solutions of \eqref{v-PME} given in the above paragraph.
\begin{enumerate}[{\rm (i)}]
\item If $v_1(\cdot ,t_1)\vartriangleright v_2(\cdot,t_1)$, then there exist $s_1, s_2$ with $t_1<s_1\leq s_2 \leq t_2$ such that $v_1(\cdot,t)\vartriangleright v_2(\cdot,t)$ for $t\in [t_1, s_1)$, $v_1(\cdot,t) \succapprox v_2(\cdot,t)$ for $t\in [s_1,s_2]$, and $v_1(\cdot,t)\succ v_2(\cdot,t)$ for $t\in (s_2, t_2)$.

\item If $r_1(t_1) = l_3(t_1)$, then either $r_1(t)=l_3(t)$ for all $t\in [t_1,t_2)$, or, there exist $s_3, s_4, s_5$ with $t_1\leq s_3 < s_4 \leq s_5 \leq t_2$ such that $r_1(t)=l_3(t)$ for $t\in [t_1,s_3]$, $v_1(\cdot,t)\vartriangleright v_3(\cdot,t)$ for $t\in (s_3, s_4)$, $v_1(\cdot,t) \succapprox v_3(\cdot,t)$ for $t\in [s_4,s_5]$, and $v_1(\cdot,t)\succ v_3(\cdot,t)$ for $t\in (s_5, t_2)$.
\end{enumerate}
\end{prop}

\noindent
In (i), the relationship $v_1(\cdot,t) \succapprox v_2(\cdot,t)$ for $t\in [s_1,s_2]$ might be true for $s_1 <s_2$, since we have now Hopf boundary lemma for degenerate equation \eqref{v-PME}. In (ii), $r_1(t)=l_3(t)$ might be true in a period after $t_1$, since the waiting times of $r_1(t)$ and $l_3(t)$ at $x=r_1 (t_1)$ may be positive.

\subsection{Cauchy problems}
Let $p(x)$ be a positive stationary solution in $\mathcal{S}$. By Proposition \ref{prop:compact-TW}, there exists $\varphi(z;c)$ satisfying \eqref{prop-q} and \eqref{homo-sub-TW} such that  $\varphi(x-c t;c)$ is a very weak subsolution of \eqref{E}, and, for any $x_0\in \R$, there holds,
\begin{equation}\label{p2>varphi}
p(x) > \varphi (x-c t+x_0;c),\qquad x,\ t\in \R.
\end{equation}
When we use the pressure functions, $p(x)$ corresponds to a positive stationary solution (denote it by $q(x)$) of \eqref{v-PME}, and $\varphi(x-c t)$ corresponds to a traveling-wave-type subsolution $\phi (x-c t)$ of \eqref{v-PME}:
$$
q (x) := \Psi(p(x)) \mbox{ for }  x\in \R,\qquad
\phi (z) := \Psi(\varphi(z)) \mbox{ for }z\in \R.
$$
$q (x)$ is $1$-periodic since we temporarily regard the spatial periodic as $1$ in this section. Denote by $H(x)$ the Heaviside function:
$$
H(x)=1\mbox{ for }x\geq 0,\qquad H(x)=0\mbox{ for } x< 0.
$$
For each integer $k$, we consider the Cauchy problem
\begin{equation}\label{uk-p-v}
\left\{
\begin{array}{ll}
v_t = B(v)v_{xx} + v_x^2 + g(x,v), & x\in \R,\ t>0,\\
v(x,0)= v_{0k}(x):= H(k-x)q^0(x), & x\in \R.
\end{array}
\right.
\end{equation}
As the Cauchy problem of \eqref{E} with initial data in $\X$, it follows from \cite{LouZhou, Sacks, Vaz-book, WuYin-book} etc. that the problem \eqref{uk-p-v} has a unique global solution, denoted by $v(x,t;k)$, which has a unique right free boundary, denoted by $r(t;k)$. In addition,
$$
v(x,t;k)>0 \mbox{ in } Q_k := \{(x,t)\mid x<r(t;k),\ t> 0\},\qquad
v(x,t;k)=0 \mbox{ in } (\R\times (0,\infty))\backslash Q_k,
$$
$v\in C(\R\times (0,\infty))\cap C^{2+\alpha,1+\alpha/2}(Q_k)$ for any $\alpha \in (0,1)$, and the Darcy's law holds:
\begin{equation}\label{bn-Darcy}
r'(t;k)= - v_x (r(t;k)-0,t;k)\geq 0,\qquad t>0.
\end{equation}
(In the next subsection we will show that $r'(t;k)$ has a positive lower bound. At present, however, we only know that $r(t;k)$ is strictly increasing in $t$. Such a conclusion for PMEs with reactions was proved in \cite[Theorem 2.7]{LouZhou}, it remains true for degenerate diffusion equations like \eqref{v-PME}.) Clearly, for any integers $k$ and $j$, there holds
\begin{equation}\label{k=j}
v(x,t;k)\equiv v(x-k+j,t;j),\qquad r(t;k)\equiv r(t;j) + k-j.
\end{equation}
Since $\phi (x-c t -k)$ is a subsolution of \eqref{uk-p-v} we have
\begin{equation}\label{varphi<vn}
\left\{
\begin{array}{l}
\phi(x-c t -k) < v(x,t;k) \mbox{ for }x\leq c t +k,\ t>0,\\
c t + k  < r(t;k) \mbox{ for }t>0.
\end{array}
\right.
\end{equation}
This implies that $r(t;k)\to \infty$ as $t\to \infty$.

\subsection{A priori estimates}\label{subsec:a-priori}

For any positive integer $n$, denote by $t_n$ the time when $v(n,t_n;0)=0$, or, equivalently, $r(t_n;0)=n$, and set
\begin{equation}\label{def:v-n}
v_n(x,t):= v(x+n,t+t_n;0),\quad r_n(t):= r(t+t_n;0)-n,\qquad  x\in \R,\ t>-t_n.
\end{equation}
Then $v_n(0,0)=0$, and
\begin{equation}\label{Darcy-bn}
r'_n(t) = -v_{nx}(r_n(t)-0,t),\qquad t>-t_n.
\end{equation}
Our purpose is to show the following convergence results:
\begin{equation}\label{conv-bn-B}
r_n(t)\to R(t) \mbox{ as }n\to \infty, \qquad \mbox{ in the topology of } C_{loc}(\R),
\end{equation}
and
\begin{equation}\label{conv-vn-V}
v_n(x,t)\to V(x,t) \mbox{ as }n\to \infty,\qquad \mbox{in the topology of } C_{loc}(\R^2) \cap C^{2+\alpha,1+\alpha/2}_{loc} (Q),
\end{equation}
where $Q:=\{(x,t)\mid x<R(t),\ t\in \R\}$ is the domain in which $V$ is positive.
Then $V(x,t)$ with free boundary $R(t)$ will be the desired periodic traveling sharp wave.
In order to show the convergence in \eqref{conv-bn-B} and \eqref{conv-vn-V}, we need some uniform-in-time a priori estimates. Due to the degeneracy of the equation, these estimates are complicated. We will specify them in several steps.

\begin{remark}\label{rem:renormalize-place}
\rm
One may try to renormalize $u(x,t)$ by considering the $d$-level set (as it was done in \cite{DGM} for RDEs): for some $d>0$, set $s_n$ the time when $u(n,s_n)=d$, and consider the renormalized sequence $\{u(x+n,t+s_n)\}$. If this sequence converges to some entire solution $\tilde{U}(x,t)$, then it might be the desired periodic traveling wave. Of course this limit is non-trivial, and the uniform estimates for the normalized solutions are simpler than that for $v_n$, since in this case, one only needs to focus on the domain where the renormalized sequence is positive and classical, and so the corresponding equations are non-degenerate. However, there is {\it a big problem left} in such a process, that is, we do not know whether $\tilde{U}(x,t)$ is a sharp wave or it is positive on the whole line $\R$. In the latter case the traveling wave is not the sharp one with minimal speed (see, for example, the PME with monostable reactions in \cite{A1}). Our renormalization process, however, can indeed leads to a sharp wave, although the subsequent uniform estimates near the free boundary are very complicated. (c.f. Remark \ref{rem:d-renorm}.)
\end{remark}

\medskip
\noindent
\underline{\it Step 1. Upper bound of $r'(t;0)$}. For any given $y>0$, define $t_y$ by $r(t_y;0)=y$. Then
$$
r(t_y;0)= y = r(t_{y+1};0) - 1 = r(t_{y+1};-1).
$$
Set $s_y := t_{y+1}-t_y$, we now show

\medskip
\noindent
{\bf Claim}: $r(s_y; -1)>0$.
\medskip

\noindent
In fact, if $r(s_y; -1)\leq 0$, then $v(x,s_y;-1)< v(x,0;0)$, and so by comparison we have $v(x,t;0)\succ v(x,t+s_y; -1)\equiv v(x+1,t+s_y; 0)$ for all $t>0$.
In particular, they have no free boundaries for all $t>0$, this, however, is a contradiction at $t=t_y$. Using this claim and using Proposition  \ref{prop:inter} we can conclude that
\begin{equation}\label{v-sharper-vt+s}
v(x,t;0) \vartriangleright v(x,t+s_y;-1) \mbox{\ \ for\ } 0<t< t_y,
\end{equation}
and
\begin{equation}\label{exact-greater}
v(x,t_y;0) \succapprox v(x,t_{y+1};-1)\equiv v(x+1,t_{y+1};0).
\end{equation}
So,
\begin{equation}\label{speed-comp}
r'(t_y;0)=-v_x (y-0,t_y;0) \geq -v_x((y+1)-0,t_{y+L};0)=r'(t_{y+1};0).
\end{equation}
(The strict inequality may not hold since we have no Hopf boundary lemma for the equation \eqref{v-PME}.)
This means that $r$ moves slower at $y+1$ than it does at $y$. Hence,
\begin{equation}\label{def-sn}
s_{n} := t_{n+1}-t_{n} \mbox{ is increasing in }n,
\end{equation}
and $t_n\to \infty$ as $n\to \infty$. For any $y>0$, denote
by $[y]$ the maximum integer not bigger than $y$, and
$\langle y \rangle := y -[y]$. When $[y]> 2$, by \eqref{speed-comp} we have
\begin{eqnarray*}
0 & \leq &  r'(t_y;0) \leq r'(t_{y-1};0)\leq \cdots \leq r'(t_{\langle y \rangle +1};0)  \\
& \leq & \bar{c} :=  \max\limits_{z\in [1,2]}  r'(t_z ;0) = \max\limits_{z\in [1,2]} \big[ -v_x(z-0,t_{z};0)\big].
\end{eqnarray*}
For any given $T_1 >0$, there holds $t_n - T_1 > t_2$ when $n$ is sufficiently large. Thus, for each large $n$ and each $t\in [-T_1,T_1]$, there exists a unique $y>2$ such that $t_y = t+t_n > -T_1 + t_n >t_2$, and so
$$
r'_n(t) = r'(t+t_n ;0) = r'(t_y; 0)\in [0, \bar{c}].
$$
By the Ascoli-Arzela lemma, a subsequence $\{r_{n_i}(t)\}$ of $\{r_n(t)\}$ converges as $i\to \infty$ to
a continuous function $R(t)$ in $C([-T_1,T_1])$. Using the Cantor's diagonal argument, we can find a subsequence of $\{r_{n_i}\}$ and a continuous function in $\R$, denote them again by $\{r_{n_i}\}$ and $R(t)$, such that, $R(0)=0$ and
\begin{equation}\label{Bni-to-B}
r_{n_i}(t) \to R(t) \mbox{ as }i\to \infty,\qquad \mbox{ in the topology of } C_{loc}(\R).
\end{equation}

\medskip
\noindent
\underline{\it Step 2. Lower bound of $v_x$ in the left neighborhood of the free boundary}.
For any given $y>3$, write $j :=[y-2]$. When $s> t_y - t_{y-j}$, as proving the claim in Step 1, one can show that $r(s;0)>j$.  Using this fact, and using a similar argument as proving \eqref{v-sharper-vt+s} and \eqref{exact-greater} we can show that
$$
v(x-j,t;0) \equiv v(x,t;j) \vartriangleright  v(x,t_y;0)\quad \mbox{\ \ for any } 0< t < t_{y-j},
$$
and
\begin{equation}\label{steeper-11}
v(x-j,t_{y-j};0) \equiv v(x,t_{y-j};j) \succapprox v(x,t_y;0).
\end{equation}
This implies that, for any $x'\in [y-1,y)$, there exists $t'\in [t_{y-j-1}, t_{y-j})$ such that the graph of $v(x,t';j)\equiv v(x-j,t';0)$ contacts that of $v(x,t_y;0)$ at exactly one point $(x',v(x',t_y;0))$, and
$$
v_x(x',t';j) \equiv v_x(x'-j,t';0) \leq v_x (x',t_y;0).
$$
Combining it together with $v_x(y-j-0,t_{y-j};0) \leq v_x (y-0,t_y;0)$ (by \eqref{steeper-11}) we have
\begin{equation}\label{gradient-upper}
\min\limits_{x'\in [y-1,y]} v_x (x',t_y;0) \geq \min\limits_{{\tiny
\begin{array}{c}
x'\in [y-1,y]\\
t'\in [t_{y-j-1},t_{y-j}]
\end{array}
}
}
v_x(x'-j,t';0)  \geq -C_1:=  \min\limits_{
{\tiny
\begin{array}{c}
z\in [1,3]\\
t\in [t_1,t_3]
\end{array}
}}
v_x(z,t;0).
\end{equation}
This give the uniform-in-time lower bound for $v_x$ in the interval $[r(t;0)-1, r(t;0)]$.

As a consequence of \eqref{gradient-upper}, we have the following estimate for $v$:
\begin{equation}\label{est-near-bt}
0\leq v(x,t;0)= -\int_{x}^{r(t;0)} v_x(x,t;0)dx \leq C_1 [r(t;0)-x],\qquad r(t;0)-1 \leq x\leq r(t;0),\ t>t_3.
\end{equation}
Note that this estimate holds only in $[r(t;0)-1, r(t;0)]$. We yet do not know the negative upper bound of $v_x$ in this interval and the positive lower bound of $v$ when $x$ is far away from $r(t;0)$, which are important when we show that the limit of $v_n$ is a non-trivial one.

\medskip
\noindent
\underline{\it Step 3. Positive lower bound of the average speed}. From \eqref{speed-comp} we know that, for any positive integer $n$, it takes more time for $r(t;0)$ to cross the interval $[n+1,n+2]$ than that to cross the interval $[n,n+1]$:
\begin{equation}\label{time-comp-n}
s_n :=  t_{n+1} -t_n  \leq s_{n+1} := t_{n+2} -t_{n+1}.
\end{equation}
Therefore, the average speed $\bar{c}_n:= \frac{1}{s_n}$ of $r(t;0)$ in the interval $[n,n+1]$ is decreasing in $n$. We now show that $\bar{c}_n$ has a positive limit:
\begin{equation}\label{barcn-lower-bound}
\bar{c}_n \searrow c^*\geq c,\quad \mbox{or, equivalently}\quad s_n \nearrow T:= \frac{1}{c^*} \mbox{\ \ as } n\to \infty,
\end{equation}
where $c$ is the speed of the traveling wave $\phi(x-c t)$ in \eqref{varphi<vn}.  In fact, if $c^*<c$, then $\bar{c}_n <c-\delta$ for some small $\delta>0$ and all large $n$ (to say, for $n\geq N$). Then by \eqref{varphi<vn} we have
$$
c t_n  <r(t_n;0)=n-N + r(t_N;0) <(c -\delta) (t_n-t_N) + r(t_N;0).
$$
This is a contradiction when $n$ is sufficiently large. This proves \eqref{barcn-lower-bound}.

\medskip
\noindent
\underline{\it Step 4. Positive lower bound of instantaneous speed $r'(t;0)$}. We will show the following lemma, which is crucial in our approach.
\begin{lem}\label{lem:positive-b'}
There exists a function $\delta_0(x)\in C^2((-\infty, 0])$ with
\begin{equation}\label{def-delta-0}
\delta_0(x) >0 \mbox{ in } (-\infty, 0), \quad  \delta_{0*}:= \liminf\limits_{x\to -\infty} \delta_0(x)>0,\quad \delta_0(0)=0,\quad \delta^* := - \delta'_0(0-0)>0
\end{equation}
such that
\begin{equation}\label{v>delta0}
\left\{
\begin{array}{l}
v(x,t;0) \geq \delta_0 (x-r(t;0)),\qquad x\leq r(t;0),\ t>t_2,\\
r'(t;0) = - v_x (r(t;0)-0,t;0) \geq \delta^* = -\delta'_0(0-0) ,\qquad t>t_2.
\end{array}
\right.
\end{equation}
\end{lem}

\begin{proof}
\underline{\it Step 1}. We first prepare an auxiliary supersolution. Define
$$
w := A'(u) \ \ \Leftrightarrow\ \ u= X(w):= (A')^{-1}(w).
$$
So $w \equiv A'(X(w))$. Differentiating this equality twice we have
$$
\frac{du}{dw} = X'(w)=\frac{1}{A''(X(w))}, \qquad \frac{d^2 u}{dw^2}= X''(w)= \frac{- A'''(X(w))}{[A''(X(w))]^3}.
$$
Substituting $u=X(w)$ into the equation \eqref{E} we have
\begin{equation}\label{equ-w=X}
w_t= w w_{xx}+ \left[1- \frac{A'(X(w))A'''(X(w))}{[A''(X(w))]^2} \right] w_x^2+ A''(X(w))f(x, X(w))[\kappa(x)- X(w)].
\end{equation}
By \eqref{F} we have $f(x,u)[\kappa(x)-u]\leq K_1 u$ for all $u\geq 0$ and some $K_1>0$. Then, by the assumption $\frac{rA''(r)}{A'(r)}\to A_*\ (r\to 0+0)$ in  \eqref{ass-A} we have
$$
A''(X(w))f(x, X(w))[\kappa(x)- X(w)] = K_1 u A''(u) \leq KA'(u)=Kw,\qquad u\geq 0,
$$
for some $K>0$. Combining this inequality with the assumption $A'(r),A''(r),A'''(r)>0$ for $r>0$ in \eqref{ass-A} we conclude that
$$
w_t \leq w w_{xx}+ w_x^2+ K w.
$$

For the upper bound $T$ of $s_n$ in \eqref{barcn-lower-bound}, we choose $\delta_*>0$ small such that
\begin{equation}\label{choice-delta1}
2 \delta_*^{1/2} \Big(3T +1\Big)^{1/2} e^{\frac{K(3T+1)}{2}} <1.
\end{equation}
For any $x_0\in \R$, we consider the function
\begin{equation}\label{def-bar-u}
\bar{w} (x,t) := \frac14 e^{K(t+1)} \left( \delta_* -\frac{(x-x_0)^2}{(t+1)e^{K(t+1)}}\right), \qquad x\in J(t),\ t>0,
\end{equation}
where
$$
J(t)= [x_0-\rho(t),\ x_0+\rho(t)]\quad \mbox{\ and\ }\quad \rho(t):= \delta_*^{1/2}(t+1)^{1/2} e^{\frac{K(t+1)}{2}},\qquad t>0.
$$
A direct calculation shows that
$$
\bar{w}_t \geq \bar{w}\bar{w}_{xx} + \bar{w}_x^2 + K\bar{w},\qquad x\in J(t),\ t>0,
$$
Thus, $\bar{w}$ is a supersolution of \eqref{equ-w=X}.
More precisely, let $\hat{w}(x,t)$ be the solution of \eqref{equ-w=X} with initial data
$$
\hat{w}(x,0) \equiv \bar{w}(x,0) = \frac14 e^K \left( \delta_* -e^{-K}(x-x_0)^2\right)_+  \in \X.
$$
Denote its free boundaries by $\hat{l}(t)$ and $\hat{r}(t)$. Then
$$
\bar{w}(x,t)\geq \hat{w}(x,t), \qquad x\in [\hat{l}(t),\hat{r}(t)],\ t\geq 0.
$$
So, the right free boundary $\bar{r}(t):= x_0 +\rho(t)$ of $\bar{w}(\cdot,t)$ satisfies
\begin{equation}\label{hatr<barr}
\hat{r}(t) \leq \bar{r}(t) = x_0 +\rho (t),\qquad t>0.
\end{equation}

\medskip
\noindent
\underline{\it Step 2}. Denote by $\hat{u}(x,t;x_0):= X(\hat{w}(x,t))$ the solution of \eqref{E} with initial data $X(\hat{w}(x,0))=X(\bar{w}(x,0))$, and denote
by $\hat{v}(x,t;x_0):= \Psi(\hat{u}(x,t;x_0))$ the corresponding pressure function. Denote it right free boundary by $\hat{r}(t;x_0)$. Here we add $x_0$ in the notations $\hat{v}(x,t;x_0)$ and $\hat{r}(t;x_0)$ to emphasize that the center of the support of $\hat{v}(x,0;x_0)$ is $x_0$. This is crucial in the subsequent discussion.
For any $n\geq 2$, $y\in [n,n+1)$ and some suitably chosen $x_0$, we will compare $v(x,t+t_{n-2};0)$ and $\hat{v}(x,t;x_0)$ on the left side of $y$.

We first choose $x_0=y-\rho(0)$, then the support of $\hat{v}(x,0;x_0)$ is $[y-2\rho(0), y]$. Since the waiting time for the right free boundary $\hat{r}(t;x_0)$ of $\hat{v}(x,t;x_0)$ is zero, $\hat{r}(t)$ crosses the point $y$ immediately at $t=0$. Of course, it is earlier than the free boundary $r(t+t_{n-2};0)$ of $v(x,t+t_{n-2};0)$.

Next, we choose $x_0\in [n-2+\rho(0),n-1+\rho(0)]$. Then by \eqref{hatr<barr} and the choice of $\delta_*$ we have
$$
\hat{r}(3T;x_0) \leq x_0 + \rho(3T) \leq  n-1 +\rho(0) + \rho(3T) < n\leq y.
$$
This means that $\hat{r}(t;x_0)$ will use time more than $3T$ to cross the point $y$. On the other hand $r(t+t_{n-2};0)$ will use time $t_y - t_{n-2}\leq s_{n-2}+s_{n-1}+s_n\leq 3T$ to cross $y$,
earlier than $\hat{r}(t;x_0)$. Since the support of $v(x,t_{n-2};0)$ lies on the left of that of $\hat{v}(x,0;x_0)$, by Proposition \ref{prop:inter} we see that
\begin{equation}\label{b-faster-r1}
v(x,t_y;0) \succ \hat{v}(x,t_y - t_{n-2};x_0).
\end{equation}

Consequently, there exists $x_0^y\in (n-2+\rho(0), y-\rho(0))$ such that \eqref{b-faster-r1} holds for all $x_0\in [n-2+\rho(0), x_0^y)$, while when we choose $x_0=x_0^y$, there holds
\begin{equation}\label{v>v1-bdry0}
v(x,t_y;0)\succapprox \hat{v}(x,t_y-t_{n-2};x_0^y),\qquad
v(y,t_y;0)= \hat{v}(y,t_y-t_{n-2};x_0^y)=0.
\end{equation}

\medskip
\noindent
\underline{\it Step 3}. Now we give lower estimates for $v(x,t_y;0)$ in $I_1(y)\cup I_2(y) \cup I_3(y)$, where
$$
I_1(y):= [y-\rho(0),y],\qquad I_2(y):= [x_0^y, y-\rho(0)],\qquad I_3(y):= [x_0^y-1, x_0^y).
$$
First, we define
$$
\tilde{\delta}_1 (x) := \frac12 \min\limits_{0\leq x_0\leq 1} \min\limits_{0\leq t\leq 3T} \hat{v}(x + \hat{r}(t;x_0),t;x_0),\qquad -\rho(0)\leq x \leq 0,
$$
and take $\delta_1(x)$ as a smoothen function of $\tilde{\delta}_0(x)$ such that
\begin{equation}\label{delta-0-property}
\delta_1(x) \leq 2 \tilde{\delta}_1 (x),\qquad \delta_1(x)>0 \mbox{ in } [-\rho(0),0),\qquad \delta'_1(0-0)<0.
\end{equation}
Then \eqref{v>v1-bdry0} implies that
\begin{equation}\label{v>v1-bdry}
v(x,t_y;0)\succapprox \delta_1 (x-y),\qquad x\in I_1(y).
\end{equation}
Next, we consider the case where $x\in I_2(y)$. In this case, by \eqref{v>v1-bdry0} we have
\begin{equation}\label{v>v1-near-bdry}
v(x,t_y;0) \geq  \min\limits_{z\in I_2(y)} \hat{v}(z,t_y-t_{n-2};x_0^y) \geq
\delta_2:= \inf\limits_{y>2} \min\limits_{z\in I_2(y)} \hat{v}(z,t_y -t_{n-2}; x_0^y).
\end{equation}
Since $t_y -t_{n-2}$ is a finite time in $(0, 3T)$ we see that $\delta_2$ is a positive constant independent of $y$. (One can prove this point by a contradiction argument.)
Finally, we consider the case where $x\in I_3(y)$. Denote
$$
D_3 := \{(z,t)\mid x_0^y-1 \leq z\leq x_0^y , 0\leq t \leq 3 T\},\qquad
D'_3 := \{(z,t)\mid 0 \leq z\leq 1, 0\leq t \leq 3 T\}.
$$
For any $x_0\in [x_0^y -1, x_0^y)$, by \eqref{b-faster-r1} we have
\begin{equation}\label{v>v1-middle}
v(x_0,t_y;0)  \geq  \min\limits_{0\leq t\leq 3T} \hat{v}(x_0,t;x_0) \geq \min\limits_{(z,t)\in D_3 } \hat{v}(z,t;z) = \delta_3 := \min\limits_{(z,t)\in D'_3 } \hat{v}(z,t;z) >0.
\end{equation}
Combining \eqref{v>v1-bdry}, \eqref{v>v1-near-bdry} together with \eqref{v>v1-middle} we conclude that, for any $y>2$, there holds
\begin{equation}\label{v>tilde-delta}
v(x,t_y;0) \succapprox \delta_0(x-y),\qquad x\in [y-\rho(0)-1,y],
\end{equation}
if we define $\delta_0(x)$ by
$$
\delta_0(x) := \left\{
\begin{array}{ll}
\delta_1(x), & x\in [-\rho(0), 0],\\
\mbox{increasing } C^2 \mbox{ function}, & x\in [-\rho(0)-\frac12, -\rho(0)],\\
\delta_4:= \frac12 \min\{\delta_1(-\rho(0)), \delta_2, \delta_3\}, & x\in [-\rho(0)-1, -\rho(0)-\frac12].
\end{array}
\right.
$$

\medskip
{\it Step 4}. Finally, by the definition of $v(x,t;k)$ and by comparison, for any positive integer $k$, we have
$$
v(x,t;0)\geq v(x,t;-k)\equiv v(x+k,t;0),\qquad x\in \R,\ t>0.
$$
Thus, the lemma is proved if we extend $\delta_0(x)$ in $(-\infty, y-\rho(0)-1]$ as $\delta_4$.
\end{proof}

\subsection{Renormalization, proofs of Theorem \ref{thm:main1} and Proposition \ref{prop:PTW-limit}}
Based on the uniform-in-time a priori estimates in the previous subsection, we now can consider the convergence in the renormalization sequence $\{v_n\}$ and $\{r_n\}$ defined by \eqref{def:v-n}.

\medskip
\noindent
\underline{\it Step 1. Convergent subsequence}. For large $n$, from the results in the previous subsection we know the following facts:
\begin{enumerate}[(a).]
\item $v_n(x,t)\geq \delta_0 (x-r_n(t))$ for $x\leq r_n(t),\ t>-t_n$;
\item $v_n(x,t)\geq \delta_4$ for all $x\leq r_n(t)-\rho(0),\ t>-t_n$;
\item $0\leq v_n(x,t)\leq C_1 [r_n(t)-x]$ for $x\in [r_n(t)-1, r_n(t)],\ t>-t_n$.
\end{enumerate}
It follows from (a), (b) and the standard parabolic theory that, for any compact domain $K\subset Q:= \{(x,t)\mid x<R(t),\ t\in \R\}$ (where $R(t)\in C(\R)$ is the limit function in \eqref{Bni-to-B}), and any $\alpha'\in (0,1)$, there holds
$$
\|v_{n_i}(x,t)\|_{C^{2+\alpha', 1+\alpha'/2}(K)}\leq C(\alpha',K),
$$
for the same subsequence $\{n_i\}$ as in \eqref{Bni-to-B}. Hence, the sequence $\{v_{n_i}\}$ has a convergent subsequence which converges in the topology $C^{2+\alpha,1+\alpha/2}(K)$ ($\alpha\in (0,\alpha')$) to some function $V(x,t)$.
Using Cantor's diagonal argument, for any $\alpha\in (0,1)$, we can find a function $V(x,t)\in C^{2+\alpha,1+\alpha/2}(Q)$ and a subsequence of $\{n\}$ (denote it again by $\{n_i\}$) such that
\begin{equation}\label{Vni-to-V-C2}
v_{n_i}(x,t)\to V(x,t) \mbox{ as }i\to \infty,\qquad \mbox{in the topology of } C^{2+\alpha,1+\alpha/2}_{loc}(Q).
\end{equation}
In addition, $V(0,0)=0$ since $v_n(0,0)=0$.

In what follows, we extend $V$ to be zero in the domain $\R^2 \backslash Q$, and write the extended function (which is defined over $\R^2$) as $V$ again.

\medskip
\noindent
\underline{\it Step 2. To show that $V$ is a very weak entire solution and $R(t)$ is its free boundary}.
Taking $n=n_i$ in (a) and (c), and taking limits as $i\to \infty$ we have
\begin{equation}\label{V-is-large-bdry}
\delta_0(x-R(t))\leq V(x,t)\leq C_1[R(t)-x] \mbox{ for } x\in [R(t)- 1, R(t)],\ t\in \R.
\end{equation}
This implies that $V(x,t)\in C(\R^2)$.  Substituting $u_{n_i}:= \psi(v_{n_i})$ into the definition of very weak solution in section 1, and taking limits as $i\to \infty$, we also see that
\begin{equation}\label{def-U-from-V}
U(x,t) := \psi(V(x,t))
\end{equation}
is a very weak solution of the equation in \eqref{E}. Equivalently, $V$ is a very weak solution of \eqref{v-PME} for all $t\in \R$, and it is classical in $Q$. This implies that $R(t)\in C^1(\R)$, it is the free boundary of $V(\cdot,t)$, and so it satisfies the Darcy's law:
\begin{equation}\label{B'=-Vx}
R'(t) = -V_x(R(t)-0,t) \geq \delta^* := -\delta'_0(0-0)>0,\qquad t\in \R.
\end{equation}

\medskip
\noindent
\underline{\it Step 3. To show that the convergence \eqref{Vni-to-V-C2} holds for the whole sequence $\{v_n\}$}. We only need to show that, if
\begin{equation}\label{vn'j-to-hatV}
v_{n'_j}(x,t)\to \widehat{V}(x,t) \mbox{ as }j\to \infty,\qquad \mbox{in the topology of } C^{2+\alpha,1+\alpha/2}_{loc}(Q),
\end{equation}
then $V(x,t)\equiv \widehat{V}(x,t)$. By contradiction, assume that, for some $s_1\in \R$,
$V(x,s_1)\not\equiv \widehat{V}(x,s_1)$. We choose a positive integer $k$ such that $-kT<s_1$.
In \eqref{exact-greater} we rewrite $x$ as $x+n$ and take $y=n-k$ for large $n$, then we have
$$
 v(x+n, t_n+(t_{n-k}-t_n);0)\geq v(x+n+1, t_{n+1}+(t_{n-k+1}-t_{n+1});0),\qquad x\in \R,
$$
that is,
$$
\tilde{v}_n(x) :=v_n(x,t_{n-k}-t_n) \geq \tilde{v}_{n+1}(x) :=v_{n+1}(x,t_{n-k+1}-t_{n+1}),\qquad x\in \R.
$$
This means that the sequence $\{\tilde{v}_n(x)\}$ is decreasing in $n$, and so it converges pointwisely as $n\to \infty$ to some function $\widetilde{V}(x)$.
On the other hand, the limits in \eqref{Vni-to-V-C2} and \eqref{vn'j-to-hatV} imply that
$$
V(x,-kT) = \widetilde{V}(x) =\widehat{V}(x,-kT),\qquad x\in \R.
$$
So, $V(x,t-kT) \equiv \widehat{V}(x,t-kT)$ for all $t>0$ since they are both solutions of
\eqref{v-PME} with the same initial data $\widetilde{V}(x)$. This leads to a contradiction when we take $t=kT +s_1>0$. This proves that the convergence \eqref{Vni-to-V-C2} holds for the whole sequence $\{v_n\}$.
Consequently, the convergence in \eqref{Bni-to-B} can also be improved to
\begin{equation}\label{Bn-to-B}
r_n(t)\to R(t) \mbox{ as }n\to \infty,\qquad \mbox{in the topology of }C_{loc}(\R).
\end{equation}

In addition, since $v(x,t;0)$ starts from $H(-x)q(x)$, for any $t>0$, $v(x,t;0)$ is {\it steeper than } any entire solution of \eqref{v-PME} lying in the range $[0,q(x)]$ (in the sense of \cite{DGM}), so does $v_n(x,t)$. In particular, $v_n(x,t_1)$ is steeper than $V(x,t_2)$. As a consequence, $V(x,t_1)$ and $V(x,t_2)$ are steeper than each other.

\medskip
\noindent
\underline{\it Step 4. To show that $V$ is a periodic traveling wave}. For any $x,t\in \R$ and any large integer $n$, we have
$$
v_{n+1}(x-1,t) =  v(x+n, t+t_n +(t_{n+1}-t_n);0) = v_n(x,t+t_{n+1}-t_n).
$$
Taking limit as $n\to \infty$ in both sides we conclude
\begin{equation}\label{PTW-V}
V(x-1,t)=V(x,t+T),\qquad x,\ t\in \R.
\end{equation}
This means that $V$ is a periodic traveling sharp wave of \eqref{v-PME}.

\begin{remark}\label{rem:c*T}\rm
By \eqref{PTW-V} we know that the traveling sharp wave $V$ has average speed $c^*:= \frac{L}{T}$, and so its free boundary $R(t)$ satisfies 
$$
R(t+kT)\equiv R(t)+kL\quad \mbox{\ \ for all integer\ \ } k.
$$ 
Thus, $R(t)-c^* t$ is a periodic function with average $0$.
\end{remark}

\medskip
\noindent
\underline{\it Step 5. To show that $V(x,t)\to q_*(x)$ as $x\to -\infty$ for some periodic stationary solution $q_*(x)$ of} \underline{\it \eqref{v-PME}.} We first prove a claim:
\begin{equation}\label{Vt>0}
V_t(x,t)>0,\qquad x,\ t\in \R.
\end{equation}
Differentiating $V(R(t),t)\equiv 0$ we have $V_t(R(t),t)= V_x^2(R(t)-0,t)>0$ by \eqref{B'=-Vx}. So, in order to prove the conclusion we need only to show that, for each $s_1\in \R$, $V_t (x,s_1)>0$ for all $x \in \R$.
Assume by contradiction that, for some $s_1$, there exist $x_1\in \R$ such that $V_t(x_1, s_1)\leq 0<V_t(R(s_1),s_1)$. Then, for small $s>0$, $V(\cdot, s_1+s)$ intersects $V(\cdot,s_1)$ at least once. Since they are steeper than each other (see the end of Step 3), we actually have $V(x,s_1)\equiv V(x,s_1 +s)$, and so $V_t(x,s_1)\equiv 0$, a contradiction.

Next, for any positive integer $k$, we set
$$
V_k (x,t) := V(x-k,t) \leq q(x),\qquad x,\ t\in \R.
$$
Since $v(x,t;1)\equiv v(x-1,t;0)\succ v(x,t;0)$, we have $v_n(x-1-k,t)\succ v_n(x-k,t)$ and so $V_{k+1}(x,t):= V(x-k-1,t)\geq V_k(x,t) := V(x-k,t)$. Thus, there exists a positive entire solution $W(x,t)$ such that
$$
V_k(x,t) \nearrow W(x,t) \mbox{ as }k\to \infty,\qquad \mbox{in the topology of } C^{2+\alpha,1+\alpha/2}_{loc} (\R^2).
$$
This also means that $V_{k+1}(x,t)\equiv V_k(x-1,t)\to W(x-1,t)$ as $k\to \infty$, and
$W(x,t)\geq \delta_4$ for all $t\in \R$ and $x\in [0,1]$. So
\begin{equation}\label{Wx-1=Wx}
W(x-1,t)\equiv W(x,t) \geq \delta_4.
\end{equation}

On the other hand, since $V$ is a periodic traveling wave, we have
$$
V_k (x-1,t) \equiv V (x-k-1,t) \equiv V(x-k, t+T) \equiv V_k(x,t+T).
$$
Taking limit as $k\to \infty$ we have $W(x-1,t)\equiv W(x,t+T)$. Combining together with \eqref{Wx-1=Wx} we have $W(x,t)\equiv W(x,t+T)$. Finally, the inequality \eqref{Vt>0} implies that  $W_t\geq 0$. So, $W_t(x,t)\equiv 0$, and $W(x,t)\equiv q_*(x)\geq \delta_4$ for some $1$-periodic stationary solution $q_*(x)$ of \eqref{v-PME}.

\begin{remark}\label{rem:ss-terrace}\rm
Note that for any give $t\in \R$ and any large $n$, $v_n(x,t)\to q(x)$ as $x\to -\infty$. However, the convergence $v_n\to V$ discussed in Step 3 is taking in the topology of $C^{2+\alpha,1+\alpha/2}_{loc}(\R^2)$.  Hence, the limit $q_*(x)$ of $V(x,t)$ as $x\to -\infty$  is not bigger than $q(x)$, but not necessarily to be $q(x)$ itself. For example, in the multistable cases, $v_n$ may be characterized by a propagating terrace. In that case, $V$ is actually the lowest traveling wave of the terrace. We will show in the next section that $q_*(x)\equiv q(x)$ for monostable, bistable and combustion equations, and no propagating terrace exists for such simple equations.
\end{remark}

\begin{remark}\label{rem:d-renorm}\rm
In our approach, we take the renormalization at the time $t_n$ when $v(n,t_n)=0$. It is natural to try to renormalize $v(x,t)$ by considering the $d$-level set (as it was  done in \cite{DGM} for RDEs): for some $d>0$, set $s_n$ the time when $v(n,s_n)=d$, and consider the renormalized sequence $v(x+n, t + s_n)$ . If this sequence converges to some entire solution
$V(x,t)$, then it can be expected to be a periodic traveling wave. Of course it is non-trivial. The problem, however, is that we do not know whether $V(x,t)$ is a sharp wave having a free boundary, or it is positive on the whole line $\R$ (c.f. Remark \ref{rem:renormalize-place}). For bistable or combustion equations, this may be clarified.
In monostable case, however, it seems difficult to distinguish the sharp wave from other traveling waves with positive profiles. Our renormalization process, however, can lead to a sharp wave.
%
\end{remark}

\medskip
\noindent
{\it Proof of Theorems \ref{thm:main1} and Proposition \ref{prop:PTW-limit}}. The conclusions follow from the construction of $V$ and the properties we proved in the above steps.
\hfill \qed

\section{Periodic Traveling Sharp Wave of Monostable, Bistable and Combustion Equations}

In this section we consider the equation \eqref{E} with special reactions as in Cases (a), (b) and (c) in section 1. Proposition \ref{prop:p>kappa} is proved in the following subsections.

\subsection{The monostable and combustion cases}
If $f$ is of Case (a), then the equation \eqref{E} has no periodic stationary solution in the range $(0,\kappa_0)$. Hence the limit $p(x)$ of $U(x,t)$ as $x\to -\infty$ must be an element in $\calS$.

If $f$ is of Case (c), then $p$ must be a positive constant in $(0,\theta]$ if it does not belong to $\calS$. Thus, the periodic traveling sharp wave $U(x,t)$ is actually a sharp wave of $u_t = [A(u)]_{xx}$. This, however, is impossible. In fact, if $U(x,t)=\varphi_0 (x-c_0 t)$ is the corresponding sharp wave with speed $c_0$ and normalized condition $U(0,0)=0$, then $\varphi_0(z)$ satisfies
\begin{equation}\label{no-reaction}
[A(\varphi_0)]''(z) + c_0 \varphi'_0 (z) =0,\qquad z\leq 0,
\end{equation}
and
$$
\varphi_0(-\infty)=p,\quad \varphi'_0(-\infty)=0,\quad \varphi_0(0)=0,\quad -[\Psi(\varphi_0)]'(0) =c_0.
$$
For $z<0$, integrating \eqref{no-reaction} over $(-\infty, z]$ we have
$$
[A(\varphi_0)]'(z)  +c_0 \varphi_0 (z) -c_0 p =0.
$$
So, $[A(\varphi_0)]'(0-0)=c_0 p$. This, however, contradicts the following fact
$$
c_0 p = [A(\varphi_0)]'(0-0) =  A'(0+0) \cdot \varphi'_0(0-0) =
[\Psi(\varphi_0)]'(0-0)\cdot \varphi_0(0-0) = 0.
$$

\subsection{The bistable case}\label{subsec:bistable-PTW}
We divide the proof into several steps.

\medskip
\noindent
\underline{\it Step 1. To seek for a transition solution}.
By $\int_u^{\kappa_0}A'(r)g_0(r)dr>0$ in \eqref{F} it is easily seen that
the equation \eqref{homo-PME-eq} with homogeneous reaction $g_0(u)$ has a compactly supported stationary solution $\varphi(x;0)$ for some $l>0$, it satisfies
$$
\varphi(x;0)>0 \mbox{ in } (-l,l),\quad \varphi(\pm l;0)=0,\quad \bar{d}:=\max\limits_{[-l,l]}\varphi(x;0)\in (\theta, \kappa_0).
$$
Now we consider the Cauchy problem of \eqref{E} with initial data $\sigma\varphi(x)\in \X$, where $\sigma>0$ is a parameter. Denote the corresponding solution by $u_{\sigma}(x,t)$.  On the one hand, the general convergence result as in \cite[Theorem 1.1]{LouZhou} holds, which guarantees that $u_{\sigma}$ converges as $t\to \infty$ to some stationary solution (denote it by $w_{\sigma}$) of the Cauchy problem of \eqref{E}. On the other hand, we see that when $\sigma\ll 1$, vanishing happens in the sense that $w_{\sigma}\equiv 0$ due to the bistable property of $f$; when $\sigma>1$, spreading happens in the sense that $w_{\sigma}\in \calS$ (In fact, for such a $\sigma$, $u_{\sigma}> \varphi$ by comparison, so $w_{\sigma}\geq \varphi$. Then by sliding $\varphi(x;0)$ below $w_{\sigma}(x)$ we see that $\varphi$ supports $w_{\sigma}$ everywhere from below. Thus $w_{\sigma}(x)\geq \bar{d}$ for all $x\in \R$, and so it is in $\calS$.) Then both $\Sigma_0:= \{\sigma>0\mid \mbox{vanishing happens for }u_{\sigma}\}$ and $\Sigma_1:= \{\sigma>0 \mid u_{\sigma} \mbox{ converes to some element in }\calS\}$ are non-empty sets. In addition, it is easy to see that both sets are open (cf. \cite{DuLou,DuMatano}). Hence $\Sigma_*:= (0,\infty)\backslash (\Sigma_0\cup\Sigma_1)$ is non-empty and closed. For any $\sigma_* \in \Sigma_*$, we see that $u_{\sigma_*}(x,t)$ converges to some positive stationary solution $w_*$, with $\sup w_*(x)>\theta$ or $w_*(x)\equiv \theta$ (since $u_{\sigma_*}$ is a solution neither spreading nor vanishing).

\medskip
\noindent
\underline{\it Step 2. Properties of $w_*$.} Denote by $Z_1(t)$ (resp. $Z_2(t)$, $Z_*$) the number of zeros of $u_{\sigma_*}(x,t)-\theta$
(resp. $u_{\sigma_*}(x,t)-\hat{p}(x)$, $w_*(x)-\theta$), where $\hat{p}$ is the periodic stationary solution of \eqref{E} in Case (b).
By the intersection number properties in subsection 3.1 we see that $Z_1(t)$ and $Z_2(t)$ are finite and decreasing in $t$.

We explain that $w_*(x) \not\equiv \theta$. Otherwise, $Z_2(t)$ will be larger and larger as $t$ increasing since $\theta$ has infinitely  many intersection points with $\hat{p}(x)$. This is a contradiction. As a consequence, we have $\sup w_*(x) >\theta$.

Next we prove $Z_*=2$. Since $Z_1(0)=2$, we see that $Z_1(t)\leq 2$ for all $t>0$, and so $Z_*\leq 2$. In fact, if $Z_*=0$ or $Z_*=1$, then $w_*(x)>\theta$ in some infinite interval $J$. Without loss of generality, we assume $J$ contains $[x_1, \infty)$.
By the assumption in the bistable Case (b), we can find a homogeneous bistable reaction (denote it again by $g_0(u)$) such that $g_0(0)=g_0(\theta)=g_0(\kappa_0)=0$ and $g'_0(\theta)>0$. Take a large interval $J_1 \subset J$ and consider the problem
\begin{equation}\label{p-J1}
\left\{
 \begin{array}{ll}
 u_t = [A(u)]_{xx}+g_0(u), & x\in J_1, t>0,\\
 u(x,t)=\theta, & x\in \partial J_1,\\
 u(x,0)=w_*(x), & x\in J_1.
\end{array}
\right.
\end{equation}
Denote its solution by $\tilde{u}(x,t)$. By comparison we first have $w_*(x)> \tilde{u}(x,t)$ in $J_1$. On the other hand, by the general convergence result, $\tilde{u}(x,t)$ converges as $t\to \infty$ to some stationary solution $\tilde{w}(x)$ of \eqref{p-J1} with $\tilde{w}(x)>\theta$ in the interior of $J_1$. Sliding $\tilde{w}(x)$ to positive infinite of the $x$-axis, we conclude that
$$
w_*(x)\geq \max\limits_{J_1} \tilde{w}(x) >\theta,\qquad x\gg x_1.
$$
This reduces to the spreading phenomena for $u_{\sigma_*}$, a contradiction.

\medskip
\noindent
\underline{\it Step 3. To show $w_*$ is a ground state with bounded support.}
From Step 2, we see that the two intersection points $(x_\pm (t),\theta)$ between $u_{\sigma_*}(\cdot,t)$ and $\theta$ satisfies
\begin{equation}\label{-M<x<M}
-M<x_-(t) <x_+(t)<M,\qquad t>0,
\end{equation}
for some $M>0$.

Now we consider an auxiliary supersolution. Choose a bistable nonlinearity $g^0(u)$ such that
$$
g^0(u)<0 \mbox{ in } (0,\theta),\quad
g^0(u)>0 \mbox{ in } (\theta, \kappa^0),\quad
(g^0)'(0)<0<(g^0)'(\theta),
$$
and
$$
f(x,u)[\kappa(x)-u]\leq g^0(u),\qquad x\in \R,\ u\geq 0.
$$
Consider the ground state solution of $[A(u)]_{xx}+g^0(u)=0,\ x\in \R$. By \cite[sections 1 and 6]{LouZhou}, this problem has a Type II ground state solution $U^0(x)$ defined over $J_0:= (b^0_1,b^0_2)$, with
$$
\max\limits_{J_0} U^0(x) >\theta,\quad U^0(x)>0 \mbox{ in } J_0,\quad U^0(x)=0 \mbox{ in }\R\backslash J_0,\quad [\Psi(U^0)]'(b^0_i) =0\ (i=1,2).
$$
Assume that $x^0$ is the biggest $\theta$-level set of $U^0(x)$. Now we choose a large $M_1>M$ and put $u^*(x):= U^0(x+x^0-M_1)|_{x\geq M_1}$ from the very beginning. Using \eqref{-M<x<M} and the comparison principle we conclude that the right free boundary $r(t)$ of $u_{\sigma_*}(\cdot,t)$ is bounded. The left boundary is also bounded. Consequently, $w_*(x)$ has compact support. Since it is a stationary solution of the Cauchy problem of \eqref{E}, not only of the the equation \eqref{E}. Therefore, $[\Psi(w^*)]'(x)=0$ on the boundary of its support, that is, $w^*$ is a Type II ground state stationary solution (cf. \cite{LouZhou} for PMEs).

\medskip
\noindent
{\it Proof of Proposition \ref{prop:p>kappa} for the bistable case}. Recall that the periodic traveling sharp wave $U(x,t)$ we obtained in the previous section moves rightward with speed bigger than $\delta^*$. Now, we put the Type II ground state solution $w_*(x)$ on the right of the free boundary $R(t)$ of $U(x,t)$, then $R(t)$ meets the left boundary of $w_*$ at some time $t_1$, then crosses the support of $w_*$ in time interval $(t_1,t_2)$, and finally passes the right boundary of $w_*$ at time $t_2$. Since both $U(x,t)$ and $w_*(x)$ solve the same equation \eqref{E}, we can use the intersection number properties in section 3 to conclude that the number $Z^0(t)$ of the intersection points between them satisfies $Z^0(t)=1$ for $t\in [t_1,t_2]$, and $Z^0(t)=0$ for $t>t_2$. This implies that the limit function $p(x)$ in Theorem \ref{thm:main1} (which is a positive periodic stationary solution) satisfies $p(x)\geq \max U^0(x) >\theta$. Consequently, $p\in \calS$.
\hfill \qed

\subsection{Exponential decay of $p(x)-U(x,t)$ as $x\to -\infty$}
{\it Proof of the inequalities \eqref{U-to-p-exp}}. The first inequality follows from the strong maximum principle. We now consider the exponential decay rate. We only need to prove it for $t\in [0,T]$:
\begin{equation}\label{U-to-p-exp-r}
p(x)-U(x,s)\leq M_* e^{\delta_* x}, \qquad x\in \R,\ s\in [0,T).
\end{equation}
In fact, if this is true, then for any $t\in \R$, we have $t=kT+s$ for some integer $k$ and some $s\in [0,T)$. Then, due to the periodicity of $U(x,t)$ we have
\begin{eqnarray*}
p(x)-U(x,t) & = & p(x)-U(x,kT+s)= p(x-kL)-U(x-kL,s)\\
& \leq & M_* e^{\delta_* (x-k L)}  \leq  M_* e^{\delta_* L}e^{\delta_* (x-c^* t)}.
\end{eqnarray*}

We now prove \eqref{U-to-p-exp-r}. For $\varepsilon_1>0$ sufficiently small, we see that \eqref{Fu<0} holds not only in $I_0$, but also in $I_{\varepsilon_1}:= [\kappa_0-\varepsilon_1, \kappa^0+\varepsilon_1]$. Since $0<p(x)-U(x,s)\to 0$ as $x\to -\infty$, there exists $X>0$ such that
$$
0<p(x)- U(x,t) <\varepsilon_1,\qquad x\leq -X,\ t\geq 0.
$$
Set
$$
\eta_1:= A(p(x))- A(U(x,t)).
$$
Then $\eta_1$ satisfies
\begin{equation}\label{eta2-p}
\left\{
   \begin{array}{ll}
   \eta_{1t}= a_1 (x, t)\eta_{1xx}+ c_1(x, t)\eta_1, & x\leq -X,\ t>0,\\
   \eta_1 (-X, t)= A(p(-X))- A(U(-X,t)), & t>0,\\
   \eta_1 (x, 0)= A(p(x))- A(U(x,0)), & x\leq -X.
   \end{array}
   \right.
\end{equation}
where for some $\theta_1, \theta_2 \in (0,1)$,
$$
a_1 (x, t)= A'(U(x,t)) \ \mbox{ and }\ c_1 (x, t)= \frac{A'(U(x,t))}{A'(\theta_1 p+ (1-\theta_1)U(x,t))} F_u (x, \theta_2 p+ (1-\theta_2)U(x,t)).
$$
By \eqref{ass-A} and \eqref{F1}, when $\varepsilon\in (0,\varepsilon_0)$ for sufficiently small $\varepsilon_0$, we have
\begin{equation}\label{def-a*a*}
\left\{
   \begin{array}{ll}
   a_* := \min\limits_{[\kappa_0-\varepsilon_0, \kappa^0+\varepsilon_0]} A'(r) \leq a_1 (x, t)\leq a^*:=\max\limits_{[\kappa_0-\varepsilon_0, \kappa^0+\varepsilon_0]} A'(r) , \\
   c_1 (x, t)\leq -\lambda := \frac{a_*}{a^*} \max\limits_{r\in[\kappa_0-\varepsilon_0, \kappa^0+\varepsilon_0]} F_r(x, r)<0.
   \end{array}
   \right.
\end{equation}
Denote $\alpha:= \Big( \frac{\lambda}{a^*}\Big)^{1/2}$. Then a direct calculation shows that
$$
\bar{\eta}_1 (x,t) := a^* \varepsilon_1 \Big( e^{\alpha (x+X)} + e^{-\lambda t} \Big),\qquad x\leq -X,\ t>0
$$
is a supersolution of \eqref{eta2-p}. Thus,
$$
p(x)-U(x,t) \leq \frac{1}{a_*} \eta_1(x,t) \leq \frac{1}{a_*}\bar{\eta}_1(x,t),\qquad x\leq -X,\ t>0.
$$

For any $s\in [0,T]$ and any $x<-4L-2X$, there exists a positive integer $k$ and a $y\in [0,2L)$ such that $x= y- 2kL$, and $y-kL<-X$. Thus
\begin{eqnarray*}
p(x) -U(x,s) & = & p(y-2kL) -U(y-2kL,s)  =  p(y-kL) - U(y-kL,s+kT) \\
 & \leq & M_4 \left(e^{\alpha(y-kL)} + e^{-\lambda (s+kT)} \right)\quad \mbox{ with }M_4:= \frac{a^* \varepsilon_1}{a_*} e^{\alpha X}\\
 & = & M_4 \left( e^{\frac{\alpha(x+y)}{2}} + e^{-\lambda s +\frac{\lambda T}{2L} (x-y)}\right)\\
 & \leq & M_5 \left( e^{\frac{\alpha}{2} x} + e^{\frac{\lambda }{2c^*}x}\right), \quad \mbox{ with } M_5 := M_4 e^{\alpha L}.
 \end{eqnarray*}
This proves \eqref{U-to-p-exp-r}. \hfill \qed

\section{Asymptotic Behavior of Spreading Solutions}
In this section, we focus on monostable, bistable and combustion equations and proof Theorem \ref{thm:Cauchy}, that is, using the periodic sharp wave to characterize the spreading solutions.
We first construct some periodic traveling wave solutions with compact supports for \eqref{E}, and then use them to give the exponential estimates for spreading solutions near $x=0$.
These estimates allow us to split a spreading solution down the middle into left and right halves, and then estimate them using super- and sub-solutions constructed by the sharp wave.

Throughout this section, we assume \eqref{F1}, and so $\calS$ has exactly one element $p(x)$. We also assume $f$ is of Case (a), (b) or (c). Thus the left infinite limit of the sharp wave $U(x,t)$ is nothing but $p(x)$.

\subsection{Periodic traveling waves with compact supports}
Besides the periodic sharp wave, we construct some periodic traveling waves with compact supported profiles, which will be used as lower solutions in comparison.

\begin{thm}\label{thm:TW1}
Suppose that the assumptions in Theorem \ref{thm:Cauchy} hold. Given any small $c>0$ and any $d\in (0,\kappa_0)$ sufficiently close to $\kappa_0$, for any large $l$, let $\varphi (x;c)$ be the solution of \eqref{TW-q} in \eqref{prop:compact-TW} whose support is $[-l,l]$. Then the equation \eqref{E} has a solution $U_r (x,t;l)$:
\begin{equation}\label{U>q}
\left\{
 \begin{array}{l}
 U_r \mbox{ is a periodic traveling wave}:\ U_r (x-L,t;l)\equiv U_r \Big(x,t+\frac{L}{c};l\Big),\\
 p(x)> U_r (x,t;l)> \varphi (x-ct;c),\qquad ct-l< x< ct+l,\\
 U_r \mbox{ has compact support}:\ U_r(ct-l,t;l)=U_r(ct+l,t;l)=0, \\
 -[\Psi(U_r)]_x(ct+l -0,t) >c.
 \end{array}
 \right.
\end{equation}
Furthermore, for any $M>0$, there holds
\begin{equation}\label{Ul-to-p}
\max\limits_{ct-M\leq x\leq ct+M} |U_r(x,t;l)-p(x)|\to 0\quad \mbox{as}\quad l\to \infty.
\end{equation}

Similarly, there is a leftward moving periodic traveling wave $U_l(x,t;l)$ satisfying $U_l(x+L,t;l)\equiv U_l\Big(x,t+\frac{L}{c};l\Big)$ and other properties as above.
\end{thm}

\begin{proof}
Consider the initial-boundary value problem
\begin{equation}\label{IBVP}
\left\{
\begin{array}{ll}
w_t = [A(w)]_{zz} + cw_z +f(z+ct,w)[\kappa(z+ct)-w], & -l<z<l,\ t>0,\\
w(-l, t)= w(l,t)=0, & t>0,\\
w(z,0)= \kappa_0, & -l\leq z\leq l.
\end{array}
\right.
\end{equation}
By the theory of nonlinear diffusion equations (cf. \cite{Sacks,Vaz-book, WuYin-book}), this problem has a unique solution $w(z,t;l)$.
Since $\varphi(z;c)$ is a time-independent subsolution of \eqref{IBVP} we have
\begin{equation}\label{wl>varphi}
w(z,t;l)> \varphi(z;c),\qquad z\in (-l, l),\ t>0.
\end{equation}
This implies that $w(z,t;l)$ is classical in $(-l,l)\times (0,\infty)$. Note that both $f(z+ct,w)$ and $\kappa(z+ct)$ are $\frac{L}{c}$-periodic functions in $t$. So we can use the zero number argument as in \cite{BPS} for time-periodic problems to conclude that $w (z,t;l)$ converges as $t\to \infty$ to some $\frac{L}{c}$-periodic solution $W(z,t;l)$, with $W(z,t;l)>\varphi(z;c)$ in $(-l, l)$. (Note that, in \cite{BPS}, the authors used the zero number argument which is applicable for non-degenerate equations as one can see in section 2.)

Define $U_r(x,t;l):= W(x-ct,t;l)$. Then $U_r (x-L,t;l)=U_r (x,t+\frac{L}{c};l)$, and so $U_r(x,t;l)$ is a periodic traveling wave solution of the equation \eqref{E}, with support $[ct-l,ct+l]$. We remark  that $U_r$ is not a solution of the Cauchy problem of \eqref{E} since it does not satisfy the Darcy's law on the right boundary $ct+l$:
 $$
 - [\Psi(U_r)]'(ct+l-0,t) \geq - [\Psi(\varphi)]'(0-0)>c.
 $$

To complete the proof of \eqref{U>q}, it is left to show $p(x) >U_r(x,t;l)$. Note that $\kappa^0$ is a supersolution of the problem \eqref{IBVP}, this implies that $w(z,t;l)\leq \kappa^0$, so does $W(z,t;l)$. Then $U_r(x,t;l)\leq \kappa^0$ for all $x\in [ct-l, ct+l],\ t\in \R$.
Let $\bar{u}(x,t;\kappa^0)$ be the solution of \eqref{E} with initial data $\kappa^0$, then $\bar{u}$ is decreasing in $t$ and converges as $t\to \infty$ to the unique element $p$ in $\calS$. This implies that $p(x)\geq U_r(x,t;l)$. The strict inequality follows from the strong maximum principle.

\medskip
Next we study the limit of $U_r(x,t;l)$ in \eqref{Ul-to-p}.
We first explain that $W(x,t;l)$ is increasing in $l$. In fact, if $l_2 >l_1$, then by comparison we have $w(z,t;l_1)\leq w(z,t;l_2)$. Taking limits as $t\to \infty$ we have $W(z,t;l_1)\leq W(z,t;l_2)$.
Note that $W(z,t;l)$ is positive and so classical except for the boundary points. Then, there exists a positive function $\mathcal{W}(z,t)$ such that, for any fixed $M>0$, there holds
\begin{equation}\label{Wl-to-calW}
\lim\limits_{l\to \infty}  \|W(\cdot,t;l)-\mathcal{W}(\cdot,t)\|_{C^2([-M,M])}=0.
\end{equation}
This implies that $\mathcal{W}(z,t)$ is a solution of the equation in \eqref{IBVP}, it is $\frac{L}{c}$-periodic in $t$. In addition, by \eqref{wl>varphi} we have
$$
\mathcal{W}(z,t)\geq \varphi(z;c),\qquad z\in [-l, l].
$$
Note that $\varphi(z;c)$ is a time-independent subsolution of the equation in \eqref{IBVP}, sliding it below $\mathcal{W}(z,t)$ along the $z$-axis we conclude that
$$
\mathcal{W}(z,t) \geq \max\limits_{[-l,l]} \varphi(z;c) \geq d,\qquad z,t\in \R.
$$
This implies that $\mathcal{U}(x,t):= \mathcal{W}(x-ct,t)$ is an entire solution of \eqref{E} and it is bigger than $d$ everywhere. By comparison we have
$$
\mathcal{U}(x,t) \geq u(x,t;d),\qquad x\in \R, t>0.
$$
Since $u(x,t;d)$ is increasing in $t$ and converges as $t\to \infty$ to the unique element $p(x)$ of $\calS$, we conclude that $\mathcal{U}(x,t)\geq p(x)$. On the other hand, the second line in \eqref{U>q} and \eqref{Wl-to-calW} imply that $\mathcal{U}(x,t)\leq p(x)$. Thus $\mathcal{U}(x,t)\equiv p(x)$, and \eqref{Wl-to-calW} can be rewritten as
\eqref{Ul-to-p}.

The leftward moving periodic wave $U_l(x,t;l)$ can be constructed similarly.
\end{proof}

\begin{remark}\rm
This theorem gives periodic traveling waves $U_r$ moving rightward. Note that, for each $U_r$, its shape changes periodically, but its support $[ct-l,ct+l]$ moves rightward with a constant speed $c$.
\end{remark}

\subsection{Exponential estimates near $x=0$}
Since the spreading solutions we are concerned have compact supports (this is different from the solutions with Heaviside or traveling wave type of profiles), to give a convergence result near the two free boundaries as in Fife and McLeod \cite{FM}, we have to cut the solution at $x=0$ into two parts, and estimate them separately. For this purpose, a good exponential convergence estimate for $u(0,t)-p(0)$ is needed. This is a curial part in such a process. We will use the periodic traveling waves with compact supports to deal with this problem.

\begin{lem}\label{lem:conve to P}
Let $u$ be a spreading solution of \eqref{E} as in Theorem \ref{thm:Cauchy}. Then for any small $c>0$, there exist $\delta> 0$ and $M$, $T_*> 0$ such that
\begin{itemize}
\item[(i)] $u(x, t)\geq p(x)- Me^{-\delta t}$ for all $x\in (-ct, ct)$ and $t\geq T_*$.
\item[(ii)] $u(x, t)\leq p(x)+ Me^{-\delta t}$ for all $t\geq T_*$.
\end{itemize}
\end{lem}

\begin{proof}
(i). For any given small $\varepsilon > 0$ and any $M>0$, by the previous lemma, there exists $l>0$ such that
\begin{equation}\label{UrUl>p-ep}
 U_r(x,t;l)\geq p(x)-\varepsilon,\quad
 U_l(x,t;l)\geq p(x)-\varepsilon\quad \mbox{ for } |x-ct|\leq M.
\end{equation}
Since $u(x,t)$ is a spreading solution, there exists $T_1>0$ such that, for all $t\geq T_1$, there holds
\begin{equation}\label{u>UrUl}
u(x,t)\geq U_r(x,0;l) \mbox{\ \ and\ \ }
u(x,t)\geq U_l(x,0;l) \quad \mbox{ for all } |x|\leq l.
\end{equation}

\noindent
{\bf Claim 1}. For any $\tau >0$, there holds,
\begin{equation}\label{u-near-p}
u(x,T_1 +\tau) \geq p(x)-\varepsilon,\qquad x\in [-c\tau, c\tau].
\end{equation}

In fact, for any $\tau>0$ and $x_1\in [0,c\tau]$ (the left side is proved similarly), set $t_1:= \frac{x_1}{c}\in [0,\tau]$, then $T'_1 := T_1 + \tau - t_1\geq T_1$. Hence \eqref{u>UrUl} holds for $t=T'_1$. By comparison we have
$$
u(x,T'_1+t)\geq U_r(x,t;l),\qquad |x-ct|\leq l,\ t>0.
$$
In particular, at $x=x_1$ and $t=t_1$ we have
$$
u(x_1,T_1+\tau) = u(x_1,T'_1+t_1) \geq U_r(x_1,t_1;l)\geq p(x_1)-\varepsilon.
$$
This proves the claim.

Now fix $\tau>0$ and consider
\begin{equation}\label{sol_in_-ct-M,+ct+M}
\left\{
   \begin{array}{ll}
   \underline{u}_{1t}= [A(\underline{u}_1)]_{xx}+ F(x,\underline{u}_1), & |x|\leq c\tau, \ t>0,\\
   \underline{u}_1(\pm c\tau, t)\equiv p(\pm c \tau)- \varepsilon, & t>0,\\
   \underline{u}_1(x, 0)= p(x)-\varepsilon, & |x|\leq c\tau,
   \end{array}
   \right.
\end{equation}
where $F(x,u) := f(x,u)[\kappa(x)-u]$. By Claim 1 and by comparison we have
$$
\kappa_0 - \varepsilon \leq \underline{u}_1(x, t) \leq u(x, T_1 +\tau+t) \leq \kappa^0+\varepsilon, \quad |x|\leq c \tau,\ t>0.
$$
Set
$$
\eta_2:= A(p(x))- A(\underline{u}_1(x, t)).
$$
Then $\eta_2$ satisfies
\begin{equation}\label{Pm-u1m}
\left\{
   \begin{array}{ll}
   \eta_{2t}= a_2 (x, t)\eta_{2xx}+ c_2 (x, t)\eta_2, & |x|\leq c\tau, \ t>0,\\
   \eta_2(\pm c\tau, t)= A(p(\pm c\tau))- A(p(\pm c \tau)- \varepsilon), & t>0,\\
   \eta_2(x, 0)= A(p(x))- A(p(x)-\varepsilon), & |x|\leq c \tau,
   \end{array}
   \right.
\end{equation}
where for some $\theta_3, \theta_4 \in (0,1)$,
$$
a_2 (x, t)= A'(\underline{u}_1) \ \mbox{ and }\ c_2 (x, t)= \frac{A'(\underline{u}_1)}{A'(\theta_3 p+ (1-\theta_3)\underline{u}_1)} F_u (x, \theta_4 p+ (1-\theta_4)\underline{u}_1).
$$
It is still difficult to calculate $\eta_2$ precisely. But we only need to give the upper bound of $\eta_2$. Let $\lambda,\ a_*$ and $a^*$ be the constants in \eqref{def-a*a*}.
Instead of \eqref{Pm-u1m} we consider the following problem
\begin{equation}\label{modi:Pm-u1m}
\left\{
   \begin{array}{ll}
   \eta_{t}= a^*\eta_{xx}-\lambda \eta, & |x|\leq c\tau, \ t>0,\\
   \eta(\pm c\tau, t)= a^* \varepsilon, & t>0,\\
   \eta(x, 0)= a^* \varepsilon, & |x|\leq c\tau.
   \end{array}
   \right.
\end{equation}
It is easily seen that both $\eta$ and $\eta_{xx}$ are positive in $Q_{\tau}:= [-c\tau,c \tau]\times (0,\infty)$. Hence, $\eta$ is a supersolution of \eqref{Pm-u1m}, and $\eta_2(x,t)\leq \eta(x,t)$ in $Q_T$.
Using a similar calculation as in the proof of Du and Lou \cite[Lemma 6.5]{DuLou}, one can show that, for small $\varepsilon_1>0$,
$$
\eta (x,s_* \tau) \leq a^*\varepsilon \left(\frac{4}{\sqrt{\pi}}+1\right)e^{- \lambda s_* \tau}, \qquad |x|\leq (1-\varepsilon_1)c\tau.
$$
where $s_* :=\frac{\varepsilon^2_1 c^2}{4a^*}$.
Using
$$
a_* [p(x)-\underline{u}_1(x, t)]\leq \eta_2(x, t):= A(p(x))- A(\underline{u}_1(x, t))\leq \eta(x, t),
$$
we conclude that
$$
u(x, T_1 +\tau + s_* \tau)\geq \underline{u}_1(x, s_* \tau) \geq p(x)-  \frac{a^*\varepsilon}{a_*} \left(\frac{4}{\sqrt{\pi}}+1\right)e^{-\lambda s_* \tau}, \qquad |x|\leq (1-\varepsilon_1)c \tau.
$$
Rewrite $t= \tau+ s_* \tau$, then we have
$$
u(x,T_1 + t)\geq p(x)- M_3 e^{-\delta t}, \qquad  |x|\leq c_1 t,\ t>0,
$$
where
$$
M_3 := \frac{a^*\varepsilon}{a_*} \left(\frac{4}{\sqrt{\pi}}+1\right),
\quad
\delta:=  \frac{\lambda s_*}{1+s_*},\quad c_1 :=
\frac{1-\varepsilon_1}{1+s_*} c,
$$
and so the estimate in (i) holds for $t> T_1$, $M_2 =M_3 e^{\delta T_1}$, and $\delta_2 =\delta$.

(ii). Using the hypothesis \eqref{F1}, the estimate from above in (ii) is much easier by considering $u(x,t;\kappa^0)$.

This completes the proof of the lemma.
\end{proof}

\subsection{Convergence of spreading solutions to the sharp wave}
In this last part, we use the sharp wave to characterize the spreading solutions.

We first explain that the additional conditions \eqref{F1} and \eqref{H} indicate some monotonicity properties for the functions $B(v)$ and $h(x,v)$ defined in subsection 2.1. Note that the condition \eqref{H} is only used here.

\begin{lem}\label{lem:increasing}
Assume \eqref{ass-A}, \eqref{F}, \eqref{F1} and \eqref{H}. Then
\begin{equation}\label{B/v-increasing}
\frac{B(v)}{v} \mbox{ is increasing in } v>0,
\end{equation}
and, there exists a small $\varepsilon_0>0$ such that
\begin{equation}\label{h/v-decreasing}
\frac{\partial}{\partial v} \left( \frac{h(x,v)}{vB(v)} \right) \leq -\lambda_0 <0 \mbox{ for } x\in \R,\ v\in J_\varepsilon:= [\Psi(\kappa_0)-\varepsilon_0, \Psi(\kappa^0)+\varepsilon_0].
\end{equation}
\end{lem}

\begin{proof}
Differentiating the right hand side of $\frac{B(v)}{v} = \frac{A'(u)}{\Psi(u)}$ and using the first condition in \eqref{H} we obtain \eqref{B/v-increasing} easily.

Denote $F(x,u):= f(x,u)[\kappa(x)-u]$. Then
$$
\frac{h(x,v)}{vB(v)} = \frac{F(x,u) \Psi'(u)}{\Psi(u) A'(u)} =
\frac{F(x,u) }{u \Psi(u)} .
$$
So,
\begin{equation}\label{dh/dv}
\frac{\partial}{\partial v} \left( \frac{h(x,v)}{vB(v)} \right)  =
\frac{\partial}{\partial u} \left( \frac{F(x,u) }{u \Psi(u)}
\right) \cdot \frac{du}{dv}  =
\frac{F_u(x,u) u \Psi(u) - F(x,u) [\Psi(u)+A'(u)]}{u (\Psi(u))^2 A'(u)} .
\end{equation}
By \eqref{F1}, there exists $\lambda>0$ such that
\begin{equation}\label{Fu<0}
F_u(x,u)< -\lambda,\qquad x\in \R,\ u\in I_0 := [\kappa_0, \kappa^0].
\end{equation}
On the other hand, when the second condition $\kappa^0-\kappa_0\ll 1$ in \eqref{H} holds, we have
$$
|F(x,u)| \ll 1,\qquad x\in \R,\ u\in I_0.
$$
Thus, the right hand side of \eqref{dh/dv} is negative when $u\in I_0$. Consequently, \eqref{h/v-decreasing} holds when $\varepsilon_0$ is sufficiently small.
\end{proof}

\medskip
\noindent
{\it Proof of Theorem \ref{thm:Cauchy}}.
The conclusions (i) and (ii) are proved in Lemma \ref{lem:unique-S} and Lemma \ref{lem:conve to P}, respectively.

(iii). We will use the generalized pressure function and prove the analogues of \eqref{u-profile-right} and \eqref{right-exp}. The analogues of \eqref{u-profile-left} and \eqref{left-exp} are proved similarly. The proof is divided into several steps.

\vskip 2mm
\noindent
\underline{\it Step 1. Notations and preliminaries.} As before, denote
$$
V(x,t) := \Psi(U(x,t))\mbox{ with } V(0,0)=0,\qquad q(x):=\Psi(p(x)).
$$
Let $R(t)$ be the free boundary of $V(x,t)$, $c^*:= \frac{L}{T}$ be its average speed. Then $|R(t)-c^* t|\leq L$. Denote
$$
\kappa'_0 := \Psi(\kappa_0),\qquad \kappa'^0:= \Psi(\kappa^0).
$$

For convenience, we list some conclusions known from above, but in the form of $v$:
\begin{enumerate}[(a)]
\item Let $\varepsilon_0\leq \frac{\kappa'_0}{2}$ be the small constant in the previous lemma. By Theorem \ref{thm:Cauchy} (ii), for any given $\varepsilon\in (0,\varepsilon_0)$, there exists a large $z_0> 0$ such that
$$
\kappa'_0 - \varepsilon \leq V(x, t)\leq \kappa'^0, \qquad x< R(t)-z_0, \ t\in \R.
$$

\item Since $V_t(x,t)>0$ for all $x\leq R(t)$, there exists $\varepsilon'_1>0$ such that
    $$
    V_t(x,t) \geq \varepsilon'_1,\qquad x\in [R(t)-z_0, R(t)],\ t>0.
    $$
\item There exists $K'>0$ such that
$$
|V_x(x, t)|\leq K', \qquad x\leq R(t),\ t\in \R.
$$

\item Since $B'(0+0)=A_*>0$, for $\varepsilon$ in (a) and any $\alpha^0>0$, there exist $B_1, B_2, B_3, B_4>0$ depending only on $\kappa'_0,\kappa'^0,\varepsilon$ and $\alpha^0$ such that
$$
\begin{array}{l}
B(r) \leq B_1,\qquad 0\leq r\leq (1+\alpha^0) \kappa'^0,\\
B(r) \geq B_2, \qquad \kappa'_0 - \varepsilon \leq r \leq (1+\alpha^0)\kappa'^0,\\
B'(r) \leq B_3,\qquad 0\leq r\leq (1+\alpha^0) \kappa'^0,\\
B'(r) \geq B_4, \qquad 0\leq r\leq (1+\alpha^0) \kappa'^0.
\end{array}
$$

\item For any $\alpha^0>0$, there exists $H_0(\alpha^0)>0$  such that
$$
|h_x(x,r)|\leq H_0,\qquad x\in \R,\ 0\leq r\leq (1+\alpha^0)\kappa'^0.
$$

\item Since $B\in C^1([0,\infty))$ with $B'(0+0)=A_*$, $h(x,\cdot)\in C^1([0,\infty))$ with $h_v(x,0+0)= \kappa(x)f_u(x,0+0)A_*$, for any $\alpha^0>0$, there exists $H_1(\alpha^0)>0$ such that
$$
h(x,\alpha v) -  \frac{\alpha h(x,v) B(\alpha v)}{B(v)} \leq H_1(\alpha^0) (\alpha-1),\qquad x\in \R, 0<v\leq \kappa'^0, 1\leq \alpha\leq 1+\alpha^0,
$$
For any $\alpha_0\in (0,1)$, there exists $H_1(\alpha_0)>0$ such that
$$
\frac{\alpha h(x,v) B(\alpha v)}{B(v)} -h(x,\alpha v) \leq H_2(\alpha_0) (1-\alpha),\qquad x\in \R, 0<v\leq \kappa'^0, 1-\alpha_0 \leq \alpha\leq 1.
$$

In fact, using the L'Hospital's rule we have
\begin{eqnarray*}
& & \lim\limits_{\alpha \to 1}\frac{h(x,\alpha v)B(v) -  \alpha h(x,v) B(\alpha v)}{(\alpha-1)B(v)} \\
& =  & \lim\limits_{\alpha\to 1} \frac{h_v(x,\alpha v)vB(v)-h(x,v)B(\alpha v) -\alpha vh(x,v)B'(\alpha v)}{B(v)} \\
& = & \frac{h_v(x,v)vB(v)-h(x,v)B(v) - vh(x,v)B'(v)}{B(v)},
\end{eqnarray*}
which is continuous and bounded in $v\in (0,v^0]$ since it tends to $0$ as $v\to 0$.

\item For $\varepsilon$ in (a), by Lemma \ref{lem:conve to P}, there exists $T'_*>0$ such that
    $$
    |v(0,t)-q(0)|\leq M'e^{-\delta' t} <\varepsilon,\qquad t\geq T'_*.
    $$
\end{enumerate}
Note that the aforementioned parameters are relatively fixed and independent. When constructing the super- and sub-solutions solutions below, additional parameters must be selected, the later may depend on the previous ones.

\vskip 2mm
\noindent
\underline{\it Step 2. Construction of supersolution}. We choose large $\alpha_1^0$ and positive integer $n_1$ such that
\begin{equation}\label{super-ini}
v(x,T'_*)< (1+\alpha_1^0) V(x,n_1 T),\qquad x\leq r(T'_*).
\end{equation}
We choose $\beta_1^0>0$ large such that
\begin{equation}\label{def-beta10}
\varepsilon'_1 \left( \frac{\beta_1^0}{\alpha_1^0} - \frac{B_3 }{B_4} -1 \right) > \kappa'^0 + H_1(\alpha_1^0),
\end{equation}
and choose
\begin{equation}\label{def-delta1}
\delta_1 := \min\left\{1, \frac{\lambda_0}{2}\right\}.
\end{equation}
Then
$$
v_1 (x,t) := \alpha_1(t) V(x,t+\beta_1(t)), \qquad x\leq R(t+\beta_1(t)),\ t\geq 0,
$$
with
$$
\alpha_1(t):= 1+\alpha_1^0 e^{-\delta_1 t},\qquad \beta_1(t) := n_1 T + \frac{\beta_1^0}{\delta_1} \big(1- e^{-\delta_1 t} \big).
$$
For simplicity, we write
$$
\V := V(x ,t+\beta_1(t)),\quad
\V_1 := V_t (x, t+\beta_1(t)),
$$
$$
\V_2 := V_x (x ,t+\beta_1(t)) \mbox{\ \ and\ \ }
\V_{22}:= V_{xx} (x,t+\beta_1(t)).
$$
Then $\V_1 = B(\V)\V_{22}+\V_2^2 +h(x,\V)$.
We divide the domain into two parts:
$$
D_1:= \{(x,t)\mid x\leq R(t+\beta_1(t))-z_0,\ t\geq 0\},
$$
$$
E_1:= \{ (x,t) \mid R(t+\beta_1(t))-z_0\leq x\leq R(t+\beta_1(t)) ,\ t\geq 0\}.
$$
and verify that $v_1$ is a supersolution respectively in $D_1$ and $D_2$.

(1) In $D_1$, a direct calculation shows that
\begin{eqnarray*}
\mathcal{N}v_1 & := &
v_{1t}- B(v_1)v_{1xx}-v_{1x}^2-  h(x, v_1)\\
& = &
-\delta_1 \alpha_1^0 e^{-\delta_1 t} \V + \alpha_1 (1+\beta_1^0 e^{-\delta_1 t}) \V_1  - \alpha_1 B(\alpha_1 \V)\V_{22} - \alpha_1^2 \V^2_2 - h(x,\alpha_1 \V)\\
& = &
-\delta_1 \alpha_1^0 e^{-\delta_1 t} \V + \alpha_1 \left[ 1+\beta_1^0 e^{-\delta_1 t} - \frac{ B(\alpha_1 \V)}{B(\V)}  \right] \V_1 \\
& & + \frac{\alpha_1 B(\alpha_1 \V)}{B(\V)} \left[ \V_1- B(\V)\V_{22} -\frac{\alpha_1 B(\V)}{B(\alpha_1 \V)} \V^2_2 -  h(x,\V)\right] \\
& & +\alpha_1 \V B(\alpha_1 \V) \left[ \frac{h(x,\V)}{\V B(\V)}  - \frac{h(x,\alpha_1 \V)}{\alpha_1 \V B(\alpha_1 \V)} \right]
\end{eqnarray*}
We know from (a) that $\kappa'_0-\varepsilon \leq \V\leq \kappa'^0$. Then using Lemma \ref{lem:increasing} and using the conclusions in (c), (d) and (e) we have
\begin{eqnarray*}
\mathcal{N}v_1 & \geq & -\delta_1 \alpha_1^0 e^{-\delta_1 t} \V  + \alpha_1  \left[ \beta_1^0 e^{-\delta_1 t} - \frac{ B(\alpha_1 \V)-B(\V)}{B(\V)}  \right] \V_1  +\lambda_0 (\alpha_1 -1)\V \\
& \geq & [\lambda_0 -\delta_1] \alpha_1^0 e^{-\delta_1 t} \V + \alpha_1 \left[ \beta_1^0  - \frac{B_3 \alpha_1^0 }{B_4} \right] e^{-\delta_1 t} \V_1\\
& \geq & [\lambda_0 -\delta_1] \alpha_1^0 \frac{\kappa'_0}{2}  e^{-\delta_1 t} >0,\qquad \mbox{ in }D_1.
\end{eqnarray*}

(2) In $E_1$, (b) implies that $\V_1\geq \varepsilon'_1$, and so as above we have
\begin{eqnarray*}
e^{\delta_1 t}\mathcal{N}v_1 & \geq & -\delta_1 \alpha_1^0 \V  + \alpha_1  \left[ \beta_1^0  - \frac{ B(\alpha_1 \V)-B(\V)}{B(\V)} e^{\delta_1 t} \right] \V_1   \\
& & - \left[ h(x,\alpha_1 \V) - \frac{\alpha_1 B(\alpha_1 \V) h(x,\V)}{ B(\V)}  \right] e^{\delta_1 t}  \\
& \geq & -\delta_1 \alpha_1^0 \V  +  \alpha_1 \left[ \beta_1^0  - \frac{B_3 \alpha_1^0 }{B_4} \right] \V_1 - H_1(\alpha_1^0) \alpha_1^0 \\
& \geq & -\delta_1 \alpha_1^0 \kappa'^0 +  \alpha_1 \left[ \beta_1^0  - \frac{B_3 \alpha_1^0 }{B_4} \right]  \varepsilon'_1 - H_1(\alpha_1^0) \alpha_1^0\\
& \geq & \left[ \beta_1^0  - \frac{B_3 \alpha_1^0 }{B_4} \right]  \varepsilon'_1 - \left[ \alpha_1^0 \kappa'^0 + H_1(\alpha_1^0) \alpha_1^0\right] >0,\qquad \mbox{ in }E_1.
\end{eqnarray*}

(3) Finally, we verify the Darcy's law on the free boundary $r_1(t):= R(S_1(t))$ with $S_1(t):= t+\beta_1(t)$:
\begin{eqnarray*}
r'_1(t)+ v_{1x}(r_1(t)-0, S_1(t))  & = & R'(S_1(t)) S'_1(t)+ \alpha_1(t) V_x(r_1(t)-0, S_1(t))\\
& = & R'(S_1(t)) [S'_1(t) - \alpha_1(t) ]  \\
& = & R'(S_1(t))(\beta_1^0 -\alpha_1^0) e^{-\delta_1 t} >0,\qquad t>0.
\end{eqnarray*}

\vskip 2mm
\noindent
\underline{\it Step 3. Construction of subsolution}. Since the solution $v(x,t)$ always has compact support, to give a lower estimate we much split it into two parts: the left part in $x\in [l(t),0]$ and right part in $x\in [0,r(t)]$.
So, when we compare a subsolution with $v$, besides the inequalities satisfied for subsolutions in the domain, we have to compare their values at $x=0$ and their slopes on the right free boundaries. The conclusion in Theorem \ref{thm:Cauchy} (iii) is crucial in  the comparison at $x=0$.

In the following, when we write $v(x,t)$ and $V(x,t)$, we mean they are defined for $x\geq 0$, each has a right free boundary and takes value $0$ beyond the boundary.

We choose $\alpha_2^0\in (0,1)$ and a nonnegative integer $n_2$ such that
\begin{equation}\label{sub-ini}
v(x,T'_*) >   (1- \alpha_2^0) V(x, n_2 T),\qquad 0\leq x\leq R(n_2 T).
\end{equation}
Choose $T'_*$ larger if necessary such that
\begin{equation}\label{p0>M'}
p(0)\alpha_2^0 >M' e^{-\delta' T'_*}.
\end{equation}
Then choose $\beta_2^0>0$ large such that
\begin{equation}\label{def-beta20}
\varepsilon'_1 (1-\alpha_2^0) \left( \beta_2^0 - \frac{B_3 \alpha_2^0 }{B_4} - \alpha_2^0\right) > \kappa'^0 \alpha_2^0 + H_2(\alpha_2^0) \alpha_2^0,
\end{equation}
and for
\begin{equation}\label{def-delta2}
\delta_2 := \min\left\{1, \delta',\lambda_0 (1-\alpha_2^0) \frac{\kappa'_0}{2}B_2  \right\},
\end{equation}
we define
$$
v_2 (x,t) := \alpha_2(t) V(x,t+\beta_2(t)), \qquad 0\leq x\leq R(t+\beta_2(t)),\ t\geq 0,
$$
with
$$
\alpha_2(t):= 1 - \alpha_2^0 e^{-\delta_2 t},\qquad \beta_2(t) := n_2 T + \frac{\beta_2^0}{\delta_2} (e^{-\delta_2 t}-1) .
$$
We now verify that $v_2$ is a subsolution.

(1) On the left boundary $x=0$, by Theorem \ref{thm:Cauchy} (ii) and \eqref{p0>M'}, for all $t\geq 0$ we have
$$
v_2 (0,t) \leq (1-\alpha_2^0 e^{-\delta_2 t}) p(0)\leq p(0)-M' e^{-\delta' (t+T'_*)}  <v(0,t+T'_*),\quad t>0.
$$

(2) On the right free boundary $r_2(t):= R(S_2(t))$ with $S_2(t):= t+\beta_2(t)$ of $v_2(x,t)$, we have
\begin{eqnarray*}
r'_2(t)+ v_{2x}(r_2(t)-0, S_2(t))  & = & R'(S_2(t)) S'_2(t)+ \alpha_2(t) V_x(r_2(t)-0, S_2(t))\\
& = & R'(S_2(t)) [S'_2(t) - \alpha_2(t) ]  \\
& = & R'(S_2(t))(-\beta_2^0 +\alpha_2^0) e^{-\delta_2 t} <0,\qquad t>0.
\end{eqnarray*}

(3) In the domain $D_2:= \{(x,t)\mid 0\leq x\leq R(t+\beta_2(t))-z_0,\ t\geq 0\}$, $V(x,t+\beta_2(t))\in [\kappa'_0-\varepsilon, \kappa'^0]$. As above we have
\begin{eqnarray*}
e^{\delta_2 t} \mathcal{N}v_2 & \leq &
 \delta_2 \alpha_2^0 V + \alpha_2  \left[ -\beta_2^0 + \frac{ B(V) - B(\alpha_2 V)}{B(V)} e^{\delta_2 t} \right] V_t  \\
 & & + \left[ \frac{\alpha_2 B(\alpha_2 V)}{B(V)} h(x,V) - h(x,\alpha_2 V) \right]e^{\delta_2 t}\\
 & \leq & \left[ \delta_2  -\lambda_0 (1-\alpha_2^0) \frac{\kappa'_0}{2}B_2  \right] \alpha_2^0 V + \alpha_2  \left[ -\beta_2^0 + \frac{B_3 \alpha_2^0 }{B_4}  \right] V_t  <0,\qquad \mbox{ in }D_2.
\end{eqnarray*}

(4) In $E_2 := \{ (x,t) \mid R(t+\beta_2(t))-z_0\leq x\leq R(t+\beta_2(t)),\ t\geq 0\}$, $V_t\geq \varepsilon'_1$, and so
\begin{eqnarray*}
e^{\delta_2 t} \mathcal{N}v_2 & \leq &
 \delta_2 \alpha_2^0 V  + \alpha_2  \left[ -\beta_2^0 + \frac{ B(V) - B(\alpha_2 V)}{B(V)} e^{\delta_2 t} \right] V_t + H_2(\alpha_2^0) \alpha_2^0  \\
 & \leq & \alpha_2^0 \kappa'^0 + H_2(\alpha_2^0) \alpha_2^0  + \alpha_2(t) \left[ -\beta_2^0 + \frac{ B_3 \alpha_2^0 }{B_4} \right] \varepsilon'_1 <0,\qquad \mbox{ in }E_2.
\end{eqnarray*}

\vskip 2mm
\noindent
\underline{\it Step 4. Convergence of $r(t)-R(t)$.} By the previous two steps and using the fact $V_t>0$ we see that
\begin{equation}\label{V2-v-V1}
(1-\alpha^*_2 e^{-\delta_2 t}) V(x,t+s_2) \leq v(x,t)\leq (1+\alpha^*_1 e^{-\delta_1 t}) V(x,t+s_1),\qquad x\in \R,\ t\geq T'_*,
\end{equation}
for some $\alpha^*_i, s_i \in \R\ (i=1,2)$, and so
$$
R(t+s_2)\leq r(t)\leq R(t+s_1),\qquad t\geq T'_*.
$$
Now for any given $s \in [s_2,s_1]$ we compare $v(x,t)$ with $V(x,t+s)$. Note that $v(\cdot,t)$  has two free boundaries and $V(\cdot,t+s)$ has only a right one. Both of them satisfy the Darcy's law on their free boundaries. Hence, using the intersection number properties Proposition \ref{prop:inter} (which is developed in \cite{LouZhou} for degenerate equations), we see that, for sufficiently large $t$, these two solutions have only simple intersections. In particular, $r(t)-R(t+s)$ does not change sign for large $t$. This fact is true for all $s\in [s_2,s_1]$, and so
$$
r(t) - R(t+t^*)\to 0\quad \mbox{\ \ as\ \ } t\to \infty,
$$
for some $t^*\in [s_2,s_1]$. By the same reason as above, $r(t)-R(t+t^*)$ does not change sign for all large $t$.

\vskip 2mm
\noindent
\underline{\it Step 5. Convergence of $v(x,t)-V(x,t+t^*)$ in $L^\infty([0,\infty))$ topology.}
For any large positive integer $k$, we define
$$
\tilde{v}_k (x,t):= v(x+kL, t+kT-t^*).
$$
Its right free boundary $\tilde{r}_k(t)$ satisfies
$$
\tilde{r}_k (t) = r(t+kT-t^*)-kL = [r(t+kT-t^*)- R(t+kT)] +R(t)\to R(t)\quad \mbox{\ \ as\ \ }k\to \infty.
$$
This implies $\tilde{r}_k(t)$ are near $R(t)$ for all large $k$. Then there is a subsequence $\{k_j\}$ of $\{k\}$ such that
$$
\tilde{v}_{k_j} (x,t)\to \widehat{V}(x,t)\quad \mbox{\ \ as\ \ }j\to \infty,
$$
in the topology of $L^\infty_{loc}(\R^2)$ and also in $C^{2+\alpha, 1+\alpha/2}_{loc}(D)$ ($\alpha\in (0,1)$), where
$\widehat{V}(x,t)$ is a very weak entire solution of \eqref{p-v} with right free boundary $R(t)$, and $D:= \{(x,t)\mid x<R(t), t\in \R\}$. (As can be seen in \cite{DB-F0, DB-F} and \cite[Theorem 7.18]{Vaz-book}, $\tilde{v}_k(x,t)$ actually have local H\"{o}lder bound, uniformly in $k$. Hence, the above convergence also holds in $C^{\alpha,\frac{\alpha}{2}}_{loc}(\R^2)$ topology, for some $\alpha\in (0,1)$.) Recall that $V(x+kL,t+kT) \equiv V(x,t)$. Then it follows from the definition of $\tilde{v}_k$ and \eqref{V2-v-V1} that
$$
V(x,t-t^* + s_2) \leq \widehat{V}(x,t)\leq V(x,t-t^*+s_1),\qquad x,t\in \R.
$$
Now we consider the intersection number between $v(x,t)$ and $V(x,t)$. Using Proposition \ref{prop:inter} again, we see that $v(x,t)-V(x,t)$ have only simple intersections for large
$t$. Therefore, as in \cite[Lemma 2.6]{DuMatano} or \cite[section 6]{Polacik-MAMS}, either the limit $\widehat{V}(x,t)$ of $\tilde{v}_k(x,t)$ is identically equal to $V(x,t)$, or $\widehat{V}(x,t)$ will not meet $V(x,t)$ on their boundaries (since they both satisfy the Darcy's law, their intersection number will decrease strictly once they meet on their boundaries). The latter, however, is obviously wrong since both of their boundaries are $R(t)$. Consequently, $\widehat{V}\equiv V$. In addition, since $V$ is unique, we actually have $\tilde{v}_k(x,t)\to V(x,t)\ (k\to \infty)$.

For any small $\varepsilon>0$, by \eqref{U-to-p-exp}, there exists $z_0>0$ such that,
$$
0<q(x) - V(x,t) \leq \varepsilon\quad \mbox{\ \ for\ \ }x\leq R(t+s_1) -z_0,\ t\in \R.
$$
Then, for any $t\gg 1$ and $x\in J_1(t):= [0, R(t+s_1)-z_0]$, by \eqref{V2-v-V1} we have
\begin{equation}\label{error-far-left}
\begin{array}{lll}
  |v(x,t)- V(x,t+t^*)| & \leq & V(x,t+s_1) - V(x,t+s_2) + \alpha^*_1 \kappa'^{0} e^{-\delta_1 t} +  \alpha^*_2 \kappa'^{0} e^{-\delta_2 t} \\
   &\leq &  (q(x)+\varepsilon) - (q(x)-\varepsilon) +  \alpha^*_1 \kappa'^{0} e^{-\delta_1 t} +  \alpha^*_2 \kappa'^{0} e^{-\delta_2 t}\\
    &\leq & 3 \varepsilon,
\end{array}
\end{equation}
provided $t$ is sufficiently large, to say, $t\geq T_1$. Next, we consider the error of $v-V$ on  the interval $J_2 (t):= [R(t+s_2)-z_0, R(t+s_1)]$ (which overlaps with $J_1(t)$ and contains $[r(t)-z_0, r(t)] \cup [R(t+t^*)-z_0, R(t+t^*)]$).
For any large $t$ with $t= kT + \hat{t}$ (where $\hat{t}\in [0,T)$) and any $x\in J_2(t)$, we have
$$
R(\hat{t}+s_2) -z_0 \leq R(t+s_2) -kL -z_0 \leq \hat{x}:= x-kL \leq R(t+s_1) -kL = R(\hat{t}+s_1).
$$
Thus,
\begin{eqnarray*}
|v(x,t-t^*)- V(x,t)| & = & |v(\hat{x}+kL,kT+\hat{t}-t^*)-V(\hat{x}+kL,kT+\hat{t})| \\
& = & |\tilde{v}_k (\hat{x},\hat{t})-V(\hat{x},\hat{t})|\to 0,\mbox{\ \ \ \ as\ \ }k\to \infty,
\end{eqnarray*}
uniformly in $(\hat{x},\hat{t})\in \hat{D}:= \{(x,t)\mid R(s_2) -z_0 \leq x\leq R(T+s_1),\ t\in [0,T]\}$. This implies that when $t$ is sufficiently large, to say $t\geq T_2$,  there holds $|v(x,t)-V(x,t+t^*)|\leq \varepsilon$ for any $x\in J_2(t)$.
Combining together with \eqref{error-far-left} we obtain the second convergence in \eqref{u-profile-right} (in the form of function $v$).

\vskip 2mm
\noindent
\underline{\it Step 6. Exponential convergence.} By \eqref{U-to-p-exp} and \eqref{V2-v-V1} we have
\begin{equation}\label{q-M<v<q}
(1-\alpha^*_2 e^{-\delta_2 t}) [q(x) -M^*_2 e^{\delta_*(x-c^* t)}]  \leq v(x,t)\leq (1+\alpha^*_1 e^{-\delta_1 t}) q(x),\qquad x\geq 0,\ t\gg 1,
\end{equation}
where $M^*_2 :=  M_* e^{-\delta_*c^* s_2}$.  By $V(x,t+s_2)\leq V(x,t+t^*)\leq V(x,t+s_1)$, \eqref{q-M<v<q} remains valid if $v(x,t)$ is replaced by $V(x,t+t^*)$. Hence, for $x\geq 0$ and $t\gg 1$, there holds
\begin{eqnarray*}
|v(x,t)-V(x,t+t^*)| & \leq & [(1+\alpha^*_1 e^{-\delta_1 t}) q(x)] - \big\{(1-\alpha^*_2 e^{-\delta_2 t}) [q(x) -M^*_2 e^{\delta_*(x-c^* t)}] \big\} \\
& \leq & M^*_2 e^{\delta_* (x-c^* t)} + (\alpha^*_1 +\alpha^*_2)\kappa'^{0} e^{-\delta_3 t}\\
& = & M^*_2 e^{\delta_* (R(t+t^*) -c^* t)} e^{-\delta_* (R(t+t^*)-x)} + (\alpha^*_1 +\alpha^*_2)\kappa'^{0} e^{-\delta_3 t}\\
& \leq & K e^{-\delta_* (R(t+t^*)-x)} + K e^{-\delta_3 t},
\end{eqnarray*}
where $\delta_3 := \min\{\delta_1,\delta_2\}$ and $K>0$ is independent of $x,t$.  This proves the estimate in \eqref{right-exp}.

This completes the proof of Theorem \ref{thm:Cauchy}. \hfill \qed

\end{document}